\documentclass[11pt]{amsart}   
\usepackage{amssymb,amscd,latexsym,verbatim}   
\usepackage{amsmath}
\usepackage[bookmarksnumbered, colorlinks, plainpages]{hyperref}
\usepackage[margin=1.4in]{geometry}
 
\usepackage{tikz}
\usepackage{epsfig,graphicx}

\begin{document}

\newcommand{\mmbox}[1]{\mbox{${#1}$}}
\newcommand{\proj}[1]{\mmbox{{\mathbb P}^{#1}}}
\newcommand{\Cr}{C^r(\Delta)}
\newcommand{\CR}{C^r(\hat\Delta)}
\newcommand{\affine}[1]{\mmbox{{\mathbb A}^{#1}}}
\newcommand{\Ann}[1]{\mmbox{{\rm Ann}({#1})}}
\newcommand{\caps}[3]{\mmbox{{#1}_{#2} \cap \ldots \cap {#1}_{#3}}}
\newcommand{\Proj}{{\mathbb P}}
\newcommand{\N}{{\mathbb N}}
\newcommand{\Z}{{\mathbb Z}}
\newcommand{\R}{{\mathbb R}}
\newcommand{\A}{{\mathcal{A}}}
\newcommand{\Tor}{\mathop{\rm Tor}\nolimits}
\newcommand{\Der}{\mathop{\rm Der}\nolimits}
\newcommand{\Exp}{\mathop{\rm exp}\nolimits}
\newcommand{\Ext}{\mathop{\rm Ext}\nolimits}
\newcommand{\Hom}{\mathop{\rm Hom}\nolimits}
\newcommand{\im}{\mathop{\rm Im}\nolimits}
\newcommand{\rank}{\mathop{\rm rank}\nolimits}
\newcommand{\supp}{\mathop{\rm supp}\nolimits}
\newcommand{\arrow}[1]{\stackrel{#1}{\longrightarrow}}
\newcommand{\CB}{Cayley-Bacharach}
\newcommand{\coker}{\mathop{\rm coker}\nolimits}
\sloppy
\newtheorem{defn0}{Definition}[section]
\newtheorem{prop0}[defn0]{Proposition}
\newtheorem{quest0}[defn0]{Question}
\newtheorem{thm0}[defn0]{Theorem}
\newtheorem{lem0}[defn0]{Lemma}
\newtheorem{corollary0}[defn0]{Corollary}
\newtheorem{example0}[defn0]{Example}
\newtheorem{remark0}[defn0]{Remark}
\newtheorem{prob0}[defn0]{Problem}
\newtheorem{conj0}[defn0]{Conjecture}

\newenvironment{defn}{\begin{defn0}}{\end{defn0}}
\newenvironment{prop}{\begin{prop0}}{\end{prop0}}
\newenvironment{quest}{\begin{quest0}}{\end{quest0}}
\newenvironment{thm}{\begin{thm0}}{\end{thm0}}
\newenvironment{lem}{\begin{lem0}}{\end{lem0}}
\newenvironment{cor}{\begin{corollary0}}{\end{corollary0}}
\newenvironment{exm}{\begin{example0}\rm}{\end{example0}}
\newenvironment{rem}{\begin{remark0}\rm}{\end{remark0}}
\newenvironment{prob}{\begin{prob0}\rm}{\end{prob0}}
\newenvironment{conj}{\begin{conj0}}{\end{conj0}}

\newcommand{\defref}[1]{Definition~\ref{#1}}
\newcommand{\propref}[1]{Proposition~\ref{#1}}
\newcommand{\thmref}[1]{Theorem~\ref{#1}}
\newcommand{\lemref}[1]{Lemma~\ref{#1}}
\newcommand{\corref}[1]{Corollary~\ref{#1}}
\newcommand{\exref}[1]{Example~\ref{#1}}
\newcommand{\secref}[1]{Section~\ref{#1}}
\newcommand{\remref}[1]{Remark~\ref{#1}}
\newcommand{\questref}[1]{Question~\ref{#1}}
\newcommand{\probref}[1]{Problem~\ref{#1}}
\newcommand{\conjref}[1]{Conjecture~\ref{#1}}

\newcommand{\std}{Gr\"{o}bner}
\newcommand{\jq}{J_{Q}}
\def\Ree#1{{\mathcal R}(#1)}


\title{Associated primes, witnesses, and omega invariants of monomial ideals}

\author[J. Miller,
M.  Nasernejad,
and \c{S}. O. Toh\v{a}neanu]{
Jacob Miller$^{1}$,
Mehrdad Nasernejad$^{2,3,*}$,
and \c{S}tefan O. Toh\v{a}neanu$^{1}$
}
\thanks{$^*$Corresponding author}

 \keywords{primary decomposition, associated prime, witness, omega invariant, irreducible decomposition, Alexander dual, v-number.\\
\indent Authors' email addresses:   jaco5247@vandals.uidaho.edu; 
m$\_$nasernejad@yahoo.com;
tohaneanu@uidaho.edu.}

\begin{abstract}
We introduce and study the omega invariant of a proper ideal in a Noetherian commutative ring, defined as the number of associated primes of the ideal. Our main objective is to investigate this invariant for monomial ideals and their powers. We characterize associated primes through monomial witnesses and provide an algorithmic procedure for constructing such witnesses from the exponent vectors of the minimal generators. These results lead to explicit formulas and bounds for the omega invariant without requiring the computation of a primary decomposition. We further establish alternative descriptions using irreducible decompositions and Alexander duality. A matrix-based approach is developed to detect associated primes of powers of monomial ideals directly from the exponent matrix of the original ideal. We also investigate the behavior of witnesses under passage from $I^n$ to $I^{n+1}$ and derive corresponding results for edge ideals of graphs. 
\end{abstract}

\date{}
\maketitle

\begin{center}
{\it
$^{1}$Department of Mathematics and Statistical Science, University of Idaho, Moscow, ID 83844 \\
$^{2}$Univ. Artois, UR 2462, Laboratoire de Math\'{e}matique de  Lens (LML),   F-62300 Lens, France \\
$^{3}$ Universit\'e  Caen Normandie, ENSICAEN, CNRS, Normandie Univ, GREYC UMR  6072, F-14000 Caen,  France
}
\end{center}

\section{Introduction}

Every proper ideal $I$ in a Noetherian commutative ring $R$ can be written as a finite intersection of primary ideals, called a primary decomposition of $I$. This theorem is due to Emanuel Lasker (1905 -- proved the result for polynomial rings and rings of formal power series) and Emmy Noether (1921 -- proved the result for any Noetherian ring), and it has been at the core of the modern development of commutative algebra (formerly known as Ideal Theory) and Algebraic Geometry. Primary ideals are generalization of irreducible ideals, and Noether proved the above theorem by showing that any proper ideal in a Noetherian ring is the intersection of finitely many irreducible ideals.

A proper ideal $I$ in a Noetherian ring $R$ can have multiple primary decompositions. Since the intersection of primary ideals is again a primary ideal, one can think of the intersection of the primary ideals the same radical ideal as just one primary ideal. Also, one can eliminate primary ideals in a decomposition that contain the intersection of the others. This way, one obtains what is commonly known as a {\em minimal} or {\em irredundant} primary decomposition of $I$. This comes with some uniqueness properties, one of them being the number of primary ideals in such a minimal primary decomposition. Though some people refer to this number as the length of the minimal primary decomposition, for reasons explained below, we call it {\em the omega invariant of $I$}, denoted $\omega(I)$.

The primary goals of these notes are to produce results and estimates of this algebraic invariant without computing the entire minimal primary decomposition, focusing on the case of monomial ideals.

The above notation and terminology are borrowed from number theory, where if $z$ is an integer, $\omega(z)$ is the number of distinct prime factors of $z$. The study of $\omega(z)$ goes back to Hardy and Ramanujan, and maybe the most familiar result is that $\omega(z)\sim\ln\ln z$. Since the theory of ideals in commutative rings was developed by Gauss as the generalization of $\mathbb Z$, to our surprise we found almost no results addressing $\omega(I)$ specifically, though somewhere in the background, we believe that the majority of manuscripts in commutative algebra or algebraic geometry are computing it in silence. More upfront, the most visible mentions of this invariant is the omega invariant of defining ideals of finite sets of points when it equals the number of points, or the omega invariant of edge ideals when it equals the number of minimal vertex covers of the graph. 
A sporadic result is \cite[Proposition 10.4]{Sw} where it is pointed out the striking difference between $\omega$ of the ideal generated by the determinants of all $r\times r$ submatrices of a matrix of indeterminates, which is a prime ideal, versus the $\omega$ of the ideal generated by the permanents of all $r\times r$ submatrices of the same matrix; in this case $\omega$ is polynomial in the size of the matrix, and if the matrix has more than three columns and more than three rows, we also have exactly one embedded component. Another interesting result occurs in \cite{Br} and says that $\omega(I^n)$ is constant not depending on $n$, for $n$ sufficiently large.

The primary ideals occurring in a minimal primary decomposition\footnote{These ideals are called {\em primary components}.} have distinct radical ideals, which are prime ideals, called {\em associated primes}. A proper ideal $I$ in a Noetherian ring $R$ has a unique set of associated primes, and $\omega(I)$ equals the cardinality of this set. So understanding $\omega(I)$ is the same as deciding which prime ideals can be associated primes of $I$. But we have a fundamental theorem that says that a prime ideal $P\subset R$ is associated prime of $I$ if and only if there exists an element $a\in R$ such that $(I:a)=P$. This element $a$ is called {\em a witness} of $I$. So the path for achieving our declared goals is clear, and reflected in the chosen title: finding the witnesses of $I$ will lead to characterizing the possible prime ideals that are associated to $I$, and from this to counting their number, hence $\omega(I)$.

\medskip

The structure of the paper is as follows. In Section \ref{sec_preliminaries} we discuss various preliminaries, including definitions of several different omega-invariants associated to the same ideal. As Remark \ref{rem_iomega_vs_omega} points out, these invariants can be quite different for the same ideal. Also, we will review a couple of well-known results regarding associated primes of ideals in general (Proposition \ref{prop_assoc}).

Results concerning witnesses, associated primes, or omega invariants for ideals in general are very hard to come by. So, very quickly we focus our attention towards these concepts specific to monomial ideals. We begin Section \ref{sec_monomial} by reviewing some of the most important properties of monomial ideals that make them so desirable to be the subject of study for any type of research. Immediately we observe that any associated prime of a monomial ideal must be generated by a subset of the variables! From this, we obtain almost with no effort an upper-bound for the omega invariant of any monomial ideal in terms of the supports of the monomial generators of that ideal (see Proposition \ref{prop_omega_monomial}); the Example \ref{exm_omega_mon} shows that the upper-bound can be strict, so we are still, long way to finding a formula.

In Subsection \ref{sec_witness} we turn our attention towards witnesses of monomial ideals. A first result we obtain is that if $P\in{\rm Ass}(R/I)$, where $I$ is a monomial ideal of $R=\mathbb K[x_1,\ldots,x_k]$, then one can obtain a witness of $I$ corresponding to $P$, by simply taking a linear combination with {\bf generic} coefficients of the minimal generators of $(I:P)$ (see, Proposition \ref{prop_witness_general_monomial}). But this result is not useful to continue our quest to find $\omega(I)$, and to our aid comes Proposition \ref{prop_omega_monomial} saying that all witnesses of a monomial ideal can be taken to be monomials. This turns out to be the essential result, because when coupled with the properties listed at the beginning of the section (especially the colon operation), one can lead to explicit listing of witnesses of monomial ideals. This is the case for Theorem \ref{thm_M_assoc} where it is constructed an explicit witness corresponding to the maximal ideal, or Theorem \ref{thm_witness_irreducible}, where explicit witnesses are constructed corresponding to associated primes produced by irreducible decompositions of the monomial ideal. Inspired by these results we prove one of the main results which is Theorem \ref{thm_witness_any_monomial} that constructs in an algorithmic way the witnesses corresponding to each of the associated primes of a monomial ideal, from the exponent vectors of the minimal monomial generators of the monomial ideal $I$. The nature of this result helps us produce a formula for the omega invariant of any monomial ideal (see, Corollary \ref{cor_omega_any_monomial}). In Example \ref{exm_assoc} we attempt to present the algorithmic way of finding all the associated primes/witnesses via the means of Theorem \ref{thm_witness_any_monomial}. Later, in Section \ref{Matrix_Method}, Definition \ref{defn_witness_exp}, Theorem \ref{thm_witness_any_monomial}, Corollary \ref{cor_omega_any_monomial}, and also Example \ref{exm_assoc} are generalized to any power $n\geq 1$ of the monomial ideal $I$, leading to the development of new concepts linked to witnesses of $I^n$.

There aren't many results (if any) that shed light on a minimal primary decomposition of a monomial ideal, but there are quite a few  that exhibits explicitly the minimal irreducible decomposition of the monomial ideal. Sections \ref{sec_quotientl} and \ref{sec_msm} are mainly present for completeness of the exposition of the paper, as we are focusing on the Alexander dual method in Section \ref{sec_Alex_dual}. A classical result (Theorem \ref{thm_Alex_dual}) shows that the minimal generators of the Alexander dual of a monomial ideal correspond to the irreducible components in the minimal irreducible decomposition of the monomial ideal. With this and the advantage of Theorem \ref{thm_witness_irreducible} that allows constructing witnesses from the irreducible components, we can produce a formula for the omega invariant of a monomial ideal by  means of the generators of the Alexander dual (see, Corollary \ref{cor_omega_ADual}). Same Theorem \ref{thm_Alex_dual} says that the omega invariant of a monomial ideal equals the number of minimal monomial generators of different supports of its Alexander dual. In Proposition \ref{prop_mu_support} we present a formula that computes the number of minimal monomial generators of different supports of any monomial ideal $I$ from the beginning of the linear strand in the graded free resolution of an immediate monomial extension of $I$. We end the Alexander dual method by obtaining a result that provides an upper-bound for the v-number of any monomial ideal.\footnote{The v-number of a non-prime homogeneous proper ideal is the smallest degree of a witness of that ideal.}

Section \ref{sec_Veronese} is the last part of Section \ref{sec_monomial}, and here we do a brief case study for a common class of monomial ideals, the Veronese type ideals. Veronese type ideals are special cases of ideals generated by fold products of linear forms, and so, by applying \cite[Theorem 2.2]{BuToXi}, one obtains a primary decomposition for these ideals, which turns out that it is minimal. The formula for the omega invariant of Veronese type ideals obtained in Theorem \ref{thm_Veronese} is an immediate application of this observation. By Lemma \ref{lem_Veronese} we can express the Alexander dual of a Veronese type ideal also as a Veronese type ideal, and we use this to express the irreducible omega for any power of a linear prime (i.e., a prime ideal generated by linear forms); see Proposition \ref{prop_iomega_veronese}.

\medskip

In Section \ref{Matrix_Method} we generalize Definition \ref{defn_witness_exp}, Theorem \ref{thm_witness_any_monomial}, Corollary \ref{cor_omega_any_monomial}, and also Example \ref{exm_assoc} to any power $n\geq 1$ of the monomial ideal $I$. The complicated route to finding associated primes, witnesses, and even omega invariants would be to find $G(I^n)$, and apply the results from Section \ref{sec_monomial}. We avoid this challenging method, by looking at the exponent matrix just of $I$; this matrix has columns the exponent vectors of the elements of $G(I)$. In Definition \ref{def:compatible_witness}, for an index set of $m$ variables, $S$, we say that a collection of $m$ (exponent) vectors is {\em $S$-dominant and compatible} if from it one can construct a candidate for a witness of $I^n$ corresponding to the prime ideal generated by the variables indexed by $S$. Actually, to make sure we can construct such a witness, one needs the extra condition that a certain integer system of inequalities, called {\em $S$-escape system}, {\bf does not have a solution}; see, Theorem \ref{thm:matrix_power_witness}. We apply this result in two examples, first recovering a result in \cite{NKA} and second revisiting an example in \cite{CNQ}: 
\begin{itemize}
    \item[-] In Example \ref{Example-1}, we show that the (irrelevant) maximal ideal is an associated prime for any power greater than 2 of powers of the cover ideal of an odd cycle graph.
    \item[-] In Example \ref{Example-2}, we show that $\langle x,y,z\rangle$ is an associated prime of $\langle x^{2m+1}z, x^my^4, x^{m+1}y^2, y^{2m+1}z\rangle^s, m\geq 2$, for any $s\geq m$.
\end{itemize}
Similarly to Corollary \ref{cor_omega_any_monomial}, in Corollary \ref{cor:omega_power_lower_bound} and, under a different formulation, in Corollary \ref{cor:empty_escape_omega}, we obtain a lower bound on $\omega(I^n)$, based on the number of subsets $S\subseteq[k]$ for which Theorem \ref{thm:matrix_power_witness} works.

In Section \ref{sec:matrix_extension}, we investigate conditions and consequences of when a witness of $I^n$ can be naturally extended to a witness of $I^{n+1}$.\footnote{``Naturally extended'' means to add a $1$ in the appropriate place to the exponent vector of $I^n$.} Theorem \ref{thm:matrix_extension} provides such conditions.

In general, if $I$ is a proper ideal of a Noetherian ring $R$, one always has a short exact sequence of $R$-modules

$$0\longrightarrow I^n/I^{n+1}\longrightarrow R/I^{n+1}\longrightarrow R/I^n\longrightarrow 0.$$ And by \cite[Exercise 9.42]{Sh}, one has

$${\rm Ass}(I^n/I^{n+1})\subseteq {\rm Ass}(R/I^{n+1})\subseteq {\rm Ass}(R/I^{n+1})\cup {\rm Ass}(I^n/I^{n+1}).$$ Hence, $$\omega(I^{n+1})\leq \omega(I^n)+\epsilon.$$ In Corollary \ref{cor:omega_monotonicity} we show the other inequality, $\omega(I^{n+1})\geq \omega(I^n)$, under some specific extension conditions. In Proposition \ref{prop:edge_extension} and Theorem \ref{thm:edge_ideal_persistence}, we translate these conditions for the case of edge ideals of simple graphs. As we mention in Remarks \ref{rem:edge_ideal_caution} and  \ref{rem:edge_extension_interpretation}, even for edge ideals these extension conditions don't come for free, and  their translation to conditions on the incidence matrix of the graph remains as a future project.

We end the manuscript with an Appendix where we present the lines of Macaulay2 code for determining witnesses based on Theorem \ref{thm_witness_any_monomial}.

\section{Preliminaries}\label{sec_preliminaries}

We start with some classical definitions and concepts. For these and other related results, our basic resource is \cite{Sh}. Also, for any positive integer $w$, throughout the notes we will use the classical notation $\{1,\ldots,w\}=[w]$.

An ideal $Q$ of a commutative ring $R$  is called {\em primary} if $Q$ is proper and, whenever $a,b\in R$ with $ab\in Q$ but $a\notin Q$, then $b\in \sqrt{Q}$. If $Q$ is primary, then $P:=\sqrt{Q}$ is a prime ideal of $R$, and we say that {\em $Q$ is a $P$-primary ideal}. Also, whenever $P'$ is a prime ideal with $Q\subseteq P'$, then $P\subseteq P'$.

Let $R$ be a commutative Noetherian ring and let $I$ be a proper ideal of $R$. Then, $I$ has a minimal primary decomposition $I=Q_1\cap\cdots\cap Q_n$ (see, \cite[Corollary 4.35]{Sh}), meaning that
\begin{itemize}
  \item[(a)] For all $i=1,\ldots,n$, $Q_i$ is a $P_i$-primary ideal.
  \item[(b)] $P_1,\ldots, P_n$ are all distinct.
  \item[(c)] For all $j=1,\ldots,n$ we have $$Q_j\nsupseteq (Q_1\cap\cdots\cap Q_{j-1}\cap Q_{j+1}\cap\cdots\cap Q_n).$$
\end{itemize}

The prime ideals in part (b) are called {\em the associated primes of $I$}, and we denote ${\rm Ass}(R/I)=\{P_1,\ldots,P_n\}$. The prime ideals in part (b) that are minimal under inclusion are the minimal primes of $I$ (see \cite[Proposition 4.24]{Sh}), and their set is denoted ${\rm Min}(I)$.

A minimal primary decomposition comes with some uniqueness properties:

\begin{itemize}
  \item[(i)] (\cite[Corollary 4.18]{Sh}) The set of associated primes is unique; hence the length $n$ is unique.
  \item[(ii)] (\cite[Theorem 4.29]{Sh}) The set of $Q_i$ $P_i$-primary ideals in a minimal primary decomposition of $I$, where $P_i\in{\rm Min}(I)$, is unique.
\end{itemize}

\medskip

An ideal $I$ of a commutative ring $R$ is called {\em irreducible} if it is a proper ideal and, whenever $I=I_1\cap I_2$ with $I_1, I_2$ ideals of $R$, then $I=I_1$ or $I=I_2$. By \cite[Proposition 4.34]{Sh}, in a commutative Noetherian ring, an irreducible ideal is primary. The converse is not true: for example, in $\mathbb K[x,y]$, the ideal $\langle x^2,y\rangle\cap \langle x,y^2\rangle$ is primary, but not irreducible.

By \cite[Proposition 4.33]{Sh}, every proper ideal in a commutative Noetherian ring can be expressed as a finite intersection of irreducible ideals. An irreducible decomposition is called {\em minimal} if no factors in the decomposition can be removed to obtain an irreducible decomposition of the same ideal. If $$I=I_1\cap\cdots\cap I_s=J_1\cap\cdots \cap J_t,$$ are two minimal irreducible decompositions for the proper ideal $I$ in a Noetherian ring, then $s=t$ and up to a reordering of the factors, $\sqrt{I_i}=\sqrt{J_i}$, for all $i=1,\ldots,s$, see, \cite[Exercise 7.19]{AtMa}.

\medskip

All of the above suggest to define the following invariants of a proper ideal $I$ in a commutative Noetherian ring $R$:

\begin{itemize}
  \item {\em the omega invariant}, $\omega(I)$, to be the length of a minimal primary decomposition of $I$.
  \item {\em the irreducible-omega invariant}, ${\rm i}\omega(I)$, to be the length of a minimal irreducible decomposition of $I$.
  \item {\em the weak-omega invariant}, ${\rm w}\omega(I)$, the number of minimal primes of $I$.
\end{itemize}

From all of the above, one has $${\rm w}\omega(I)\leq \omega(I)\leq {\rm i}\omega(I).$$

Note that for a square-free monomial ideal $I$, since any minimal irreducible decomposition is also a minimal primary decomposition, and since we also have ${\rm Min}(I)={\rm Ass}(R/I)$, one has $${\rm w}\omega(I)=\omega(I)={\rm i}\omega(I).$$ In this case, any primary component is a prime ideal generated by variables indexed by a minimal vertex cover of an associated simplicial complex (see \cite[Theorem 1.7]{MiSt}).

\begin{rem}\label{rem_iomega_vs_omega} The omega and irreducible-omega invariants can be quite different (see Proposition \ref{prop_iomega_veronese}). In general, if $Q$ is a $P$-primary ideal in a commutative Noetherian ring, obviously $\omega(Q)=1$, but as \cite[Exercises 8.29 and 8.30]{Sh} show $${\rm i}\omega(Q)=\dim_{R_P/PR_P}((QR_P:PR_P)/QR_P),$$ where $R_P$ is the usual notation for the localization of $R$ at the prime ideal $P$.

The weak-omega invariant can be quite different from the other omega invariants, as well. Take the following $I\subset\mathbb K[x_1,\ldots,x_k]$, $\mathbb K$ being a field of characteristic zero,

$$I=\langle x_1\rangle\cap \langle x_1,x_2\rangle^2\cap \langle x_1,x_2,x_3\rangle^3\cap\cdots\cap\langle x_1,\ldots,x_k\rangle^{k}.$$ Obviously, $\sqrt{I}=\langle x_1\rangle$, so ${\rm Min}(I)=\{\langle x_1\rangle\}$, and consequently, ${\rm w}\omega(I)=1$.

Any power of an ideal generated by a subset of variables is a primary ideal, and since the primary decomposition presented is minimal, $\omega(I)=k$. As we mentioned earlier, in Proposition \ref{prop_iomega_veronese} we present the complicated formula for the irreducible-omega for each primary component. Maybe some irreducible components are common to different primary components of $I$, but nonetheless, $\displaystyle {\rm i}\omega(I)\geq {\rm i}\omega(\langle x_1,\ldots,x_k\rangle^k)=\sum_{i=0}^{k-2}(-1)^i{{k}\choose{i}}{{k^2-i(k+1)}\choose{k-1}}$, after removing the zero terms in the formula in Proposition \ref{prop_iomega_veronese}.
\end{rem}

\subsection{Associated primes of ideals}\label{sec_assoc_ideals}

The most essential result is \cite[Proposition 8.22]{Sh}:

\begin{thm}\label{thm_essential} Let $R$ be a commutative Noetherian ring and let $I\subset R$ be a proper ideal. Let $P$ be a prime ideal of $R$. Then, $P\in{\rm Ass}(R/I)$ if and only if there exists $a\in R$ such that $(I:a)=P$.
\end{thm}

\begin{cor}\label{cor_useful} Let $R$ be a commutative Noetherian ring and let $I\subset R$ be a proper ideal. Let $P\in{\rm Ass}(R/I)$. Then $(I:P)\neq I$.
\end{cor}
\begin{proof} Let $P\in{\rm Ass}(R/I)$, by Theorem \ref{thm_essential}, $(I:a)=P$, for some $a\in R$. Obviously, $a\notin I$, since otherwise, $1$ would belong to $P$. Also, $a\alpha\in I$ for every $\alpha\in P$, and so $a\in (I:P)$. Since $I\subset (I:P)$, we obtain that $(I:P)\neq I$.
\end{proof}

Next, we will take a closer look at the associated primes $P$ of any ideal $I$, trying to avoid the search for those elements $a\in R$ with $(I:a)=P$. First, we have a useful lemma.

\begin{lem}\label{lem_colon} Let $I$ and $J$ be two ideals in a commutative Noetherian ring $R$. Then,
$$(I:J)=(I:(I:(I:J))).$$ Furthermore, if $(I:(I:J))=L$, then $(I:(I:L))=L$ as well.
\end{lem}
\begin{proof} Denote $K:=(I:J)$. We need to prove that $K=(I:(I:K))$.

To show $K\subseteq (I:(I:K))$, consider $g\in K$ and $f\in (I:K)$, both arbitrary. Then, $fg\in I$. So, $g\cdot(I:K)\subseteq I$, hence $g\in (I:(I:K))$.

For the other inclusion, we do the following. Let $h\in J$ be arbitrary. Since $K=(I:J)$, then for any $g\in K$, we have $hg\in I$. So $h\cdot K\subseteq I$, leading to $h\in I:K$. Therefore, $J\subseteq (I:K)$.

Since for any ideals $K_1,K_2,K_3$ with $K_1\subseteq K_2$ one has $(K_3:K_2)\subseteq (K_3:K_1)$, from above we obtain

$$(I:(I:K))\subseteq (I:J)=K.$$

So, the equality is proved.

\medskip

If $(I:(I:J))=L$, then $(I:L)=(I:(I:(I:J)))=(I:J)$, and hence $(I:(I:L))=L$.
\end{proof}

The above general lemma proves the following well-known result in the literature.

\begin{prop}\label{prop_assoc} Let $R$ be a commutative Noetherian ring, and let $I$ be a proper ideal of $R$. Let $P$ be a prime ideal of $R$. 

Then, $P\in{\rm Ass}(R/I)$ if and only if $P=(I:(I:P))$.
\end{prop}
\begin{proof} Let $P$ be a prime ideal with $P\in{\rm Ass}(R/I)$. By Theorem \ref{thm_essential}, $P=(I:g)$ for some $g\in R$. Hence, by Lemma \ref{lem_colon}, $P=(I:(I:P))$.

\medskip

Conversely, let $P$ be a prime ideal of $R$ with $(I:(I:P))=P$. Suppose $(I:P)$ is minimally generated by $h_1,\ldots,h_m$. Then,

$$P=(I:(I:P))=(I:\langle h_1,\ldots,h_m\rangle)=\bigcap_{i=1}^m(I:h_i).$$

Since $P$ is a prime ideal, by \cite[Corollary 3.56]{Sh}, $P=(I:h_{i_0})$, for some $i_0\in [m]$. Then, by Theorem \ref{thm_essential}, $P\in{\rm Ass}(R/I)$.
\end{proof}

\section{Witnesses and omega invariants of monomial ideals}\label{sec_monomial}

Let $R:=\mathbb K[x_1,\ldots,x_k]$ and let $I$ be a monomial ideal. The goal of this section is to make estimates on the omega invariants of monomial ideals in terms of classical invariants of these ideals. The following are some basic properties that are at the core of any method of computing a primary decomposition of a monomial ideal.

\begin{itemize}
    \item[(i)] If $I=\langle f_1,\ldots,f_s\rangle$ and $J=\langle g_1,\ldots,g_t\rangle$ are two monomial ideals, then $I\cap J=\langle \{{\rm lcm}(f_i,g_j)|1\leq i\leq s, 1\leq j\leq t\}\rangle$.
    \item[(ii)] If $I=\langle f_1,\ldots,f_s\rangle$ and if $g\in R$ is a monomial, then
    $$(I:g)=\left\langle\frac{f_1}{\gcd(f_1,g)},\ldots, \frac{f_s}{\gcd(f_s,g)}\right\rangle.$$
    \item[(iii)] If $I=\langle fg\rangle+J$ is a monomial ideal with $fg$ a minimal generator of $I$ with $\gcd(f,g)=1$, then $I=(\langle f\rangle +J)\cap(\langle g\rangle+J)$.
    \item[(iv)] For a monomial $M$ in $R:=\mathbb K[x_1,\ldots,x_k]$, define {\em the support of $M$} as ${\rm supp}(M)=\{i\,\mid \, x_i|M\}$, and for a monomial ideal, $I$, define ${\rm supp}(I)$ to be the union of the supports of the minimal monomial generators of $I$. With this notation, a monomial ideal is primary iff its radical is generated by $x_j$, for all $j\in {\rm supp}(I)$.
    \item[(v)] By \cite[Theorem 5.27]{MiSt}, every monomial ideal has a minimal irreducible decomposition with monomial components, and therefore it has a minimal primary decomposition with monomial components.
\end{itemize}

Properties (iv) and (v) above imply that any associated prime of a monomial ideal is of the form $\langle x_{i_1},\ldots,x_{i_m}\rangle$.

\medskip

As an immediate corollary to Proposition \ref{prop_assoc}, we have the following upper-bound for the omega invariant of a monomial ideal.

\begin{prop}\label{prop_omega_monomial} Let $I$ be a monomial ideal in $R:=\mathbb K[x_1,\ldots,x_k]$. Suppose $I$ is minimally generated by monomials $f_1,\ldots,f_s$. Let $\Lambda_I$ be the set of subsets $\{i_1,\ldots, i_u\}$ of $\{1,\ldots,k\}$ such that $(I:\langle x_{i_1},\ldots,x_{i_u}\rangle)\neq I$. Then,

$$\omega(I)\leq|\{A\in\Lambda_I\,|\, A\cap {\rm supp}(f_j)\neq \emptyset \mbox{ for all } j=1,\ldots,s\}|.$$    
\end{prop}
\begin{proof} Let $\frak p\in{\rm Ass}(R/I)$. Then, $\frak p=\langle x_{i_1},\ldots,x_{i_m}\rangle$ and $I\subseteq \frak p$. Therefore, $f_j\in\frak p$, for all $j=1,\ldots,s$. Since $\frak p$ is prime, then $\{i_1,\ldots,i_m\}\cap{\rm supp}(f_j)\neq\emptyset$, for all $j=1,\ldots,s$.
Also, note that by Corollary \ref{cor_useful},
we have $\{i_1,\ldots,i_m\}\in\Lambda_I$.
\end{proof}

\begin{exm} \label{exm_omega_mon} The inequality in Proposition \ref{prop_omega_monomial} can be strict. Consider the following  monomial ideal
\[
I=\langle x_1^3,x_1^2x_2^2,x_1^2x_3^2\rangle \subset R=K[x_1,x_2,x_3].
\]
Then we  have the minimal primary decomposition
\[
I=\langle x_1^2\rangle\cap\langle x_1^3,x_2^2,x_3^2\rangle.
\]
Hence, $\omega(I)=2.$ The minimal generators have supports
\[
\operatorname{supp}(x_1^3)=\{1\},\qquad
\operatorname{supp}(x_1^2x_2^2)=\{1,2\},\qquad
\operatorname{supp}(x_1^2x_3^2)=\{1,3\}.
\]
With \cite{GrSt} we can check that the subsets satisfying
\[
A\cap\operatorname{supp}(f_j)\neq\varnothing
\qquad\text{for all }j=1,2,3
\]
and belonging to $\Lambda_I$ are
\[
\{1\},\quad \{1,2\},\quad \{1,3\},\quad \{1,2,3\}.
\]
Thus $
\left|
\left\{
A\in\Lambda_I:
A\cap\operatorname{supp}(f_j)\neq\varnothing
\text{ for all }j
\right\}
\right|=4.
$
Consequently, we get $\omega(I)=2<4.$
\end{exm}





\medskip

\subsection{Witnesses.} \label{sec_witness} The element $a$ in Theorem \ref{thm_essential} is called {\em a witness} of the associated prime $P$ of $I$. Since our goal is to count the different associated primes of an ideal $I$, it is of utmost importance to have an idea what are the possible witnesses.

For a monomial ideal, the results in this regard are a bit easier to come about. First, we present a simpler theoretical result.

\begin{prop}\label{prop_witness_general_monomial}
    Let $P$ be an associated prime of the monomial ideal $I\subset R:=\mathbb K[x_1,\ldots,x_k]$, where $\mathbb K$ is a field of characteristic zero. Suppose $(I:P)$ is generated by monomials $h_1,\ldots,h_t$. Let $h:=c_1h_1+\cdots+c_th_t$, where $c_1,\ldots,c_t$ are random (generic) elements of $\mathbb K$. Then $(I:h)=P$.
\end{prop}
\begin{proof}
Obviously, $P\subseteq (I:h)$, as $h_iP\subset I$, for all $i=1,\ldots,t$. 

For the other inclusion, suppose there exists $N$ be an element of $(I:h)\setminus P$. Then, $Nh=c_1Nh_1+\cdots+c_tNh_t$ is in the monomial ideal $I$. But since $c_1,\ldots,c_t$ are taken to be generic, no monomial term of $c_iNh_i$ will cancel a monomial term of $c_jNH_j$, for $i\neq j$. Since $I$ is a monomial ideal,  it follows that  for all $i=1,\ldots, t$, every monomial term of $Nh_i$ is in $I$, and hence, $Nh_1,\ldots,Nh_t$ are all in $I$. So $N\in(I:\langle h_1,\ldots,h_t\rangle)=(I:(I:P))$. But, by Proposition \ref{prop_assoc}, since $P$ is an associated prime, $(I:(I:P))=P$; a contradiction with $N\notin P$. So $(I:h)\subseteq P$.
\end{proof}

By \cite[Corollary 1.3.10]{HeHi}, the witness of an associated prime of a monomial ideal can be a monomial itself.

\begin{prop}\label{prop_assoc_monom} If $I$ is a proper monomial ideal of $R$, and $P\in{\rm Ass}(R/I)$, then there exists a monomial $M\in R$ such that $P=(I:M)$.
\end{prop}

\vskip .2in

Proposition \ref{prop_assoc_monom} and property (ii) above are leading to a method of finding witnesses of monomial ideals of various associated primes. Suppose $I\subset R:=\mathbb K[x_1,\ldots,x_k]$ is a monomial ideal minimally generated by monomials $f_1,\ldots,f_s$, and that $P=\langle x_{i_1},\ldots,x_{i_m}\rangle\in {\rm Ass}(R/I)$. Then, we have $$\langle x_{i_1},\ldots,x_{i_m}\rangle=(\langle f_1,\ldots,f_s\rangle:M)=\left\langle\frac{f_1}{\gcd(f_1,M)},\ldots, \frac{f_s}{\gcd(f_s,M)}\right\rangle.$$

\subsubsection{The maximal ideal as an associated prime.} We will denote $\frak m:=\langle x_1,\ldots,x_k\rangle$ the maximal ideal of $R:=\mathbb K[x_1,\ldots,x_k]$. For a monomial $f\in R$, denote $d_i(f):=\deg_{x_i}(f)$ to be the exponent of the variable $x_i$ in the monomial $f$. Also, the (unique) set of minimal monomial generators of a monomial ideal $I$ is often denoted with $G(I)$.

\cite[Theorem 2.7 and proof of Theorem 2.3]{NaRa} gives the complete characterization of the case when $\frak m$ is the maximal ideal (or any embedded associated primes).        

\begin{thm}\label{thm_M_assoc} {\rm (\cite[Theorem 2.7, proof of Theorem 2.3]{NaRa})} Let $I\subset R$ be a monomial ideal with minimal generating set of monomials $$G(I)=\{\underbrace{x_1^{r_{1,1}}\cdots x_k^{r_{1,k}}}_{f_1},\ldots,\underbrace{x_1^{r_{s,1}}\cdots x_k^{r_{s,k}}}_{f_s}\}, s\geq k.$$ For each $j=1,\ldots,k$, let $C_j:=\{i_t\mid d_j(f_{i_t})=\max\{d_j(f_{i_1}),\ldots,d_j(f_{i_k})\}\}$. 

Then, $\frak{m}\in {\rm Ass}(R/I)$ if and only if there exist $i_1,\ldots,i_k\in [s]$ with the following properties:
\begin{itemize}
\item[(1)] $|C_j|=1$, for all $j=1,\ldots,k$.
\item[(2)] $C_i\cap C_j=\emptyset$, for any $i\neq j$.
\item[(3)] For any $j\in [s]\setminus\{i_1,\ldots,i_k\}$, $f_j\nmid x_1^{d_1(f_{i_1})-1}\cdots x_k^{d_k(f_{i_k})-1}$.
\end{itemize}

Furthermore, the monomial $x_1^{d_1(f_{i_1})-1}\cdots x_k^{d_k(f_{i_k})-1}$ is a witness of $\frak m$.    
\end{thm}

\medskip

\subsubsection{Witnesses in irreducible decompositions.} As we will see later in the paper, handling minimal irreducible decompositions is much easier than minimal primary decompositions. This is because an irreducible component of a monomial ideal is generated by pure powers of a subset of the variables.

\cite[Theorem 3.1]{AmSe} presents  the following result regarding the witnesses of associated primes, with the caveat that one needs to compute the minimal irreducible decomposition of the monomial ideal.

\begin{thm}\label{thm_witness_irreducible} {\rm (\cite[Theorem 3.1]{AmSe})} Let $I\subset R=\mathbb K[x_1,\ldots,x_k]$ be a monomial ideal with minimal irreducible decomposition $\displaystyle I=\bigcap_{i=1}^rQ_i$. Let $P=\langle x_{i_1},\ldots,x_{i_m}\rangle\in {\rm Ass}(R/I)$, and let $\langle x_{i_1}^{a_1},\ldots,x_{i_m}^{a_m}\rangle$ be a $P$-primary component in the minimal irreducible decomposition above. Let $\{s_1,\ldots,s_{k-m}\}=[k]\setminus\{i_1,\ldots,i_m\}$ and for any $j=1,\ldots,k-m$, choose any integers $b_j\geq \max\{d_{s_j}(f)\mid f\in G(I)\}$.

Then, the monomial $$v:=x_{i_1}^{a_1-1}\cdots x_{i_m}^{a_m-1}x_{s_1}^{b_1}\cdots x_{s_{k-m}}^{b_{k-m}}$$ satisfies $(I:v)=P$.
\end{thm}

\medskip

\subsubsection{The general case.}\label{sec_general_witness} Theorems \ref{thm_M_assoc} and \ref{thm_witness_irreducible} are the result of deeper understanding of the following analysis.

\begin{thm} \label{thm_witness_any_monomial}Let $I\subset R=\mathbb K[x_1,\ldots,x_k]$ be a monomial ideal with $G(I)=\{f_1,\ldots,f_s\}$. Then, 
\begin{itemize}
\item[(a)] $P:=\langle x_{i_1}\rangle$ is an associated prime of $I$ if and only if $x_{i_1}\mid f_j$, for all $j=1,\ldots,s$.

\item[(b)] $P:=\langle x_{i_1},\ldots,x_{i_m}\rangle, 2\leq m\leq k$ is an associated prime of $I$ if and only if there exists $m$ elements of $G(I)$ labeled $f_{i_1},\ldots,f_{i_m}$ satisfying the following two conditions:
\begin{itemize}
    \item[(i)] For $u=2,\ldots,m$, $$d_{i_u}(f_{i_u})\geq \max\{d_{i_u}(f_{i_v})+1\mid v\neq u, v\in[m]\}.$$
    \item[(ii)] For any $v\in[s]\setminus\{i_1,\ldots,i_m\}$ there exists $u(v)\in[m]$ such that $$d_{i_{u(v)}}(f_v)\geq d_{i_{u(v)}}(f_{i_{u(v)}}).$$    
\end{itemize}
\end{itemize}
Furthermore, in these conditions, for both cases (a) and (b), a monomial witness $M$ of $P$ can be constructed by taking $d_{i_u}(M)=d_{i_u}(f_{i_u})-1$, and $d_w(M)=\max\{d_w(f_{\ell})\mid\ell=1,\ldots,s\}$, for all $w\in[k]\setminus\{i_1,\ldots,i_m\}$.
\end{thm}
\begin{proof} (a) If $\langle x_{i_1}\rangle\in {\rm Ass}(R/I)$, then $I\subseteq \langle x_{i_1}\rangle$, and so, $x_{i_1}\mid f_j$, for all $j=1,\ldots,s$.

Conversely, let $P=\langle x_{i_1}\rangle$ and suppose $x_{i_1}\mid f_j$ for all $j\in[s]$. After some reordering of the elements of $G(I)$, let $i_1\in[s]$ be such that $\delta:=d_{i_1}(f_{i_1})=\min\{d_i(f_{\ell})\mid \ell\in[s]\}$. Obviously $\delta\geq 1$. 

Let $\displaystyle M:=x_{i_1}^{\delta-1}\prod_{j\neq i_1}x_j^{\Delta_j}$, where, for $j\in[k]\setminus\{i_1\}$, $\Delta_j:=\max\{d_j(f_\ell)\mid\ell\in[s]\}$. Then,

$$(I:M)=\langle \{x_{i_1}^{d_{i_1}(f_{\ell})-\delta+1}\mid \ell\in[s]\}\rangle=\langle x_{i_1}\rangle=P.$$ So $P\in {\rm Ass}(R/I)$.

\medskip

(b) For convenience and without any loss of generality, in the proof below we will substitute $\{i_1,\ldots,i_m\}$ with $\{1,\ldots,m\}$.

\medskip

First, suppose $P=\langle x_1,\ldots,x_m\rangle\in{\rm Ass}(R/I)$.

Let $M\notin I$ be a monomial such that $(I:M)=P$, and, from property (ii) at the beginning of the section, suppose

$$f_i=x_i\gcd(f_i,M), \mbox{ for all } i=1,\ldots,m.$$

So, for all $i=1,\ldots,m$, we have $x_iM=\alpha_if_i$ with $\gcd(x_i,\alpha_i)=1$. Therefore, for all $i=1,\ldots,m$, $f_i=x_if_i'$ for some monomial $f_i'\in R$ and hence

$$M=\alpha_1f_1'=\alpha_2f_2'=\cdots=\alpha_mf_m'.$$

Since $x_1$ doesn't divide $\alpha_1$, then for all $j=2,\ldots,m$, $d_1(f_j')=d_1(f_j)\leq d_1(f_1')=d_1(f_1)-1$.

So we can conclude that for $i=1,\ldots,m$ we have condition (i) satisfied: $$d_i(f_i)-1\geq d_i(f_j), \mbox{ for all } j\in[m]\setminus\{i\}.$$ Therefore

$$M=x_1^{d_1(f_1)-1}\cdots x_m^{d_m(f_m)-1}N, \mbox{ where } N\in\mathbb K[x_{m+1},\ldots,x_k].$$

On the other hand, for each $i=1,\ldots,s$, we have the factorization $f_i=g_i\tilde{f}_i$, with $g_i\in\mathbb K[x_1,\ldots,x_m]$ and $\tilde{f}_i\in\mathbb K[x_{m+1},\ldots,x_k]$. And therefore, for each $i=1,\ldots,m$, $\tilde{f}_i\mid N$.


In order for $(I:M)=\langle x_1,\ldots,x_m\rangle$, we must have  for any $\ell=m+1,\ldots,s$, $${\rm supp}\left(\frac{f_{\ell}}{\gcd(f_{\ell},M)}\right)\cap\{1,\ldots,m\}\neq\emptyset.$$ Let $i_{\ell}$ be an element of the above intersection, and denote $N_{\ell}:=\gcd(f_{\ell},M)$.

We have $$f_{\ell}=\beta_{\ell}N_{\ell}x_{i_{\ell}}, f_{\ell}=\gamma_{\ell}N_{\ell}, M=\delta_{\ell}N_{\ell}, \mbox{ with } \beta_{\ell}, \gamma_{\ell}, \delta_{\ell} \mbox{ monomials, and } \gcd(\gamma_{\ell},\delta_{\ell})=1.$$

So $$\delta_{\ell}f_{\ell}=\beta_{\ell}x_{i_{\ell}}M \mbox{ with } \gcd(\delta_{\ell},\beta_{\ell}x_{i_{\ell}})=1.$$ But then, we obtain condition (ii): $$d_{i_{\ell}}(f_{\ell})\geq d_{i_{\ell}}(f_{i_{\ell}}).$$

\medskip

For the converse, suppose that $f_1,\ldots,f_m\in G(I)$ satisfy the two conditions
\begin{itemize}
    \item[(i)] For $i=1,\ldots,m$, $$d_{i}(f_{i})\geq \max\{d_{i}(f_{j})+1\mid j\neq i, j\in[m]\}.$$
    \item[(ii)] For any $\ell\in[s]\setminus[m]$ there exists $i_{\ell}\in[m]$ such that $$d_{i_{\ell}}(f_{\ell})\geq d_{i_{\ell}}(f_{i_{\ell}}).$$   
\end{itemize}

For $w=m+1,\ldots,k$, denote $D_w:=\max\{d_w(f_{\ell})\mid \ell=1,\ldots,s\}$, and consider the monomial

$$M:=x_1^{d_1(f_1)-1}\cdots x_m^{d_m(f_m)-1}x_{m+1}^{D_m}\cdots x_k^{D_k}.$$ We will show that $(I:M)=\langle x_1,\ldots,x_m\rangle$.

From condition (i), for any $i=1,\ldots,m$, $\displaystyle \gcd(f_i,M)={f_i}/{x_i}$. So for any $i=1,\ldots,m$, $f_i/\gcd(f_i,M)=x_i$.

From condition (ii), for any $\ell=m+1,\ldots,s$, since there is $i_{\ell}\in[m]$ with $d_{i_{\ell}}(f_{\ell})\geq d_{i_{\ell}}(f_{i_{\ell}})>d_{i_{\ell}}(M)\geq 1$, then $ x_{i_{\ell}}$ divides $\displaystyle\frac{f_{\ell}}{\gcd(f_{\ell},M)}$.

Therefore, $(I:M)=\langle x_1,\ldots,x_m\rangle$.
\end{proof}

\begin{defn}\label{defn_witness_exp} A $k$-tuple ${\bf e}:=(e_1,\ldots,e_k)$ of nonnegative integers is called {\em a witness exponent vector} of the monomial ideal $I\subset\mathbb K[x_1,\ldots,x_k]$, if $M_{\bf e}:=x_1^{e_1}\cdots x_k^{e_k}$ is the monomial witness constructed in Theorem \ref{thm_witness_any_monomial}; i.e., there exists $1\leq m\leq k$ and $\{f_{i_1},\ldots,f_{i_m}\}\subseteq G(I)=\{f_1,\ldots,f_s\}$, with
\begin{itemize}
    \item[(a)] For $u=1,\ldots,m$, $e_{i_u}=d_{i_u}(f_{i_u})-1$, and for $w\in [k]\setminus\{i_1,\ldots,i_m\}$, $e_w=\max\{d_w(f_{\ell})\mid \ell=1,\ldots,s\}$.
    \item[(b)] For $u=1,\ldots, m$, $e_{i_u}\geq\max\{d_{i_u}(f_{i_v})\mid v\neq u, v\in[m]\}$.
    \item[(c)] For any $v\in[s]\setminus\{i_1,\ldots,i_m\}$, there exists $u(v)\in[m]$ such that $d_{i_{u(v)}}(f_v)\geq e_{i_{u(v)}}+1$.
\end{itemize}
\end{defn}

With the above definition we have the following result calculating the omega invariant of any monomial ideal:

\begin{cor}\label{cor_omega_any_monomial} Let $I\subset R=\mathbb K[x_1,\ldots,x_k]$ be a monomial ideal. Then,

$$\omega(I)=|\{(I:M_{\bf e})\mid {\bf e} \mbox{ is a witness exponent vector of } I\}|.$$
    
\end{cor}
\medskip

Next, we apply Theorem \ref{thm_witness_any_monomial} for an example. In Example \ref{Example-2} initially we will produce similar, yet more challenging trial-and-error computations to provide candidates for witnesses, but in a more general setup, and with a different, more encompassing approach.

\begin{exm}\label{exm_assoc}
Consider the ideal $$I=\langle x^3y^2z^3, x^2y^3z, xy^4, z^4\rangle\subset R:=\mathbb K[x,y,z].$$

First note that since $\sqrt{I}=\langle z,xy\rangle$, then every prime ideal containing $I$ must contain $z$ and must contain either $x$ or $y$. Hence the possible prime ideals containing $I$ include
\[
\langle x,z\rangle,\qquad
\langle y,z\rangle,\qquad
\langle x,y,z\rangle,
\]
but not
\[
\langle x\rangle,\qquad
\langle y\rangle,\qquad
\langle z\rangle,\qquad
\langle x,y\rangle.
\]

In the tables of exponents below, we label each row as the possible $f_{i_1},\ldots,f_{i_m}$ that first satisfy condition (i), and after we chose these we comment if condition (ii) is satisfied for the other rows (labeled here with $A$ and $B$).

\medskip

\noindent $\bullet$ Checking if $\langle x,z\rangle$ is an associated prime.

\medskip

{\em 1st try:}

\medskip

\begin{tabular}{c|c|c|c|l}
elements on $G(I)$ & $x$ & $y$ & $z$ & comments\\\hline
$f_x$ & 3 & 2 & 3 & $d_x(f_x)>d_x(f_z)$\\\hline
$A$ & 2 & 3 & 1 & {\bf not possible:} $d_x(A)<d_x(f_x)$ and $d_z(A)<d_z(f_z)$\\\hline
$B$ & 1 & 4 & 0 & \\\hline
$f_z$ & 0 & 0 & 4& only option: $d_z(f_z)>d_z(f_x)$    
\end{tabular}

\medskip

{\em 2nd try:}

\medskip

\begin{tabular}{c|c|c|c|l}
elements on $G(I)$ & $x$ & $y$ & $z$ & comments\\\hline
$A$ & 3 & 2 & 3 & \\\hline
$f_x$ & 2 & 3 & 1 & $d_x(f_x)>d_x(f_z)$\\\hline
$B$ & 1 & 4 & 0 & {\bf not possible:} $d_x(B)<d_x(f_x)$ and $d_z(B)<d_z(f_z)$\\\hline
$f_z$ & 0 & 0 & 4& \shortstack[l]{ only option: $d_z(f_z)>d_z(f_x)$, $d_x(A)>d_x(f_x)$,\\ $d_z(B)<d_z(f_x)$ }
\end{tabular}

\medskip

{\em 3rd try:}

\medskip

\begin{tabular}{c|c|c|c|l}
elements on $G(I)$ & $x$ & $y$ & $z$ & comments\\\hline
$A$ & 3 & 2 & 3 & $d_x(A)\geq d_x(f_x)$\\\hline
$B$ & 2 & 3 & 1 & $d_x(B)\geq d_x(f_x)$\\\hline
$f_x$ & 1 & 4 & 0 & only option: $d_x(f_x)>d_x(f_z)$\\\hline
$f_z$ & 0 & 0 & 4& only option: $d_z(f_z)>d_z(f_x)$    
\end{tabular}

\medskip

As we can see in the {\em 3rd try}, we can construct $M=x^{1-1}z^{4-1}y^{\max\{2,3,4,0\}}=y^4z^3$, obtaining $$(I:M)=\langle x^3,x^2,x,z\rangle=\langle x,z\rangle,$$ and consequently, $\langle x,z\rangle\in {\rm Ass}(R/I).$

\medskip

\noindent $\bullet$ Checking if $\langle y,z\rangle$ is an associated prime.

\medskip

{\em 1st try:}

\medskip

\begin{tabular}{c|c|c|c|l}
elements on $G(I)$ & $x$ & $y$ & $z$ & comments\\\hline
$f_y$ & 3 & 2 & 3 & $d_y(f_y)>d_y(f_z)$\\\hline
$A$ & 2 & 3 & 1 & $d_y(A)\geq d_y(f_y)$\\\hline
$B$ & 1 & 4 & 0 & $d_y(B)\geq d_y(f_y)$\\\hline
$f_z$ & 0 & 0 & 4& only option: $d_y(f_z)<d_y(f_y)$, $d_z(f_z)>d_z(f_y)$    
\end{tabular}

\medskip

As we can see in the {\em 1st try}, we can construct $M=y^{2-1}z^{4-1}x^{\max\{3,2,1,0\}}=x^3yz^3$, obtaining $$(I:M)=\langle y,y^2,y^3,z\rangle=\langle y,z\rangle,$$ and consequently, $\langle y,z\rangle\in {\rm Ass}(R/I)$.

\medskip

\noindent $\bullet$ Checking if $\langle x,y,z\rangle$ is an associated prime.

\medskip

{\em 1st try:}

\medskip

\begin{tabular}{c|c|c|c|l}
elements on $G(I)$ & $x$ & $y$ & $z$ & comments\\\hline
$f_x$ & 3 & 2 & 3 & $d_x(f_x)>d_x(f_y), d_x(f_z)$\\\hline
$f_y$ & 2 & 3 & 1 & $d_y(f_y)>d_y(f_x),d_y(f_z)$\\\hline
$A$ & 1 & 4 & 0 & $d_y(A)\geq d_y(f_y)$\\\hline
$f_z$ & 0 & 0 & 4& only option: $d_z(f_z)>d_z(f_x), d_z(f_y)$    
\end{tabular}

\medskip

As we can see in the {\em 1st try}, we can construct $M=x^{3-1}y^{3-1}z^{4-1}=x^2y^2z^3$, obtaining $$(I:M)=\langle x,y,y^2,z\rangle=\langle x,y,z\rangle,$$ and consequently, $\langle x,y,z\rangle\in {\rm Ass}(R/I)$. Of course, we could have applied Theorem \ref{thm_M_assoc}, which is a special case of Theorem \ref{thm_witness_any_monomial}.

To conclude the example, $$\boxed{{\rm Ass}(R/I)=\{\langle x,z\rangle, \langle y,z\rangle, \langle x,y,z\rangle\}.}$$
\end{exm}

\bigskip

Below we present well-known methods to obtain the irreducible decomposition of a monomial ideal. Quite immediately from these methods one can obtain the irreducible omega invariant. 

For a vector ${\bf a}=(a_1,\ldots,a_k)\in\mathbb N_0^{k}$ (called {\em exponent vector}), denote ${\bf x}^{\bf a}:=x_1^{a_1}\cdots x_k^{a_k}$ and $\frak m^{\bf a}:=\langle \{x_i^{a_i}~|~a_i\neq 0\}\rangle$.

\subsection{The quotient method}\label{sec_quotientl} This first method is the most natural one when it comes to obtaining a primary decomposition for monomial ideals. Because it is a recursive method, heavily dependent on the example at hand, we will not discuss a qualitative estimate for the omega invariant (or the irreducible omega invariant) of a monomial ideal in this situation. But for completeness, presenting this method is a must.

Let $I=\langle f_1,\ldots,f_s\rangle \subset {\mathbb K}[x_1,\ldots,x_k]$ be a monomial ideal minimally generated by $\{f_1,\ldots,f_s\}$. A primary decomposition of $I$ can be found by the following procedure (see \cite[Algorithm 3.28]{Va}):
\begin{itemize}
    \item If there is a simple power of each $x_i$ in $I$, then $I$ is primary.
    \item If $I$ is not primary, let $x_j^d$ be the largest power of $x_j$ such that $x_j\mid f_\ell$ for some $1\leq \ell\leq s$.  Then
    $$I=(I,x_j^d)\cap (I:x_j^d).$$
    \item Repeat with $(I,x_j^d)$ and $(I:x_j^d)$ to get a primary decomposition of $I$.
\end{itemize}

\begin{exm}\label{exm-quotient}
    Let $I=\langle x^2y, xy^3,x^2z,yz^2 \rangle \subset {\mathbb Q}[x,y,z]$. Starting with $x^2$, we have
    $$I=(I,x^2)\cap (I:x^2) = \langle x^2,xy^3,yz^2 \rangle \cap \langle y,z \rangle .$$
    Apply the algorithm again to the first ideal, using $y^3$:
    $$I=\langle x^2,y^3,yz^2\rangle \cap \langle x,z^2 \rangle \cap \langle y,z\rangle .$$
    Finally, apply again to the first ideal with $z^2$ to get a primary decomposition of $I$:
    $$I=\langle x^2,y^3,z^2\rangle \cap \langle x^2,y\rangle \cap \langle x,z^2 \rangle \cap \langle y,z\rangle.$$
\end{exm}

\medskip

\subsection{The Alexander dual method}\label{sec_Alex_dual}

Let ${\bf a}=(a_1,\ldots,a_k)$ and ${\bf b}=(b_1,\ldots,b_k)$ be two exponent vectors. We say that ${\bf a}\preceq {\bf b}$ if $a_i\leq b_i$, for all $i=1,\ldots,k$.

If ${\bf a}\preceq {\bf b}$, define ${\bf b}\setminus{\bf a}:=(c_1,\ldots,c_k)$ to be the exponent vector with the $i$-th entry
$$c_i:=\left\{
                          \begin{array}{ll}
                            b_i+1-a_i, & \hbox{if $a_i\geq 1$,} \\
                            0, & \hbox{if $a_i=0$.}
                          \end{array}
                        \right.$$

\medskip

Let $I:=\langle {\bf x}^{{\bf a}_1},\ldots, {\bf x}^{{\bf a}_s}\rangle$ be a monomial ideal minimally generated by $s$ monomials (i.e., $\mu(I)=s$). Let ${\bf b}$ be an exponent vector with ${\bf b}\succeq{\bf a}_j$, for all $j=1,\ldots,s$. {\em The Alexander dual of $I$ w.r.t. ${\bf b}$} is the ideal

$$I^{[{\bf b}]}:=\bigcap_{j=1}^s\frak m^{{\bf b}\setminus{\bf a}_j}.$$

Then, we have the following theorem that relates the irreducible omega invariant to the minimum number of generators of the Alexander dual.

\begin{thm}\label{thm_Alex_dual} (\cite[Theorem 5.27]{MiSt}) Let $I$ be a monomial ideal such that each of its minimal generators divide a monomial ${\bf x}^{\bf b}$, for some exponent vector ${\bf b}$. Then, $I$ has a unique minimal irreducible decomposition given by

$$I=\bigcap\{\frak m^{{\bf b}\setminus{\bf a}}\,|\, {\bf x}^{\bf a} \mbox{ is a minimal generator of } I^{[{\bf b}]}\}.$$ Therefore,

$${\rm i}\omega(I)=\mu(I^{[{\bf b}]}).$$    
\end{thm}

The following example illustrates how Theorem~\ref{thm_Alex_dual} can be applied in practice.
\begin{exm}
 Let $R=K[x,y,z]$ and consider the
monomial ideal
$I=\langle x^2y,xz^2,y^2z\rangle$.
Take $\mathbf{b}=(2,2,2)$. Since each minimal generator of $I$
divides $x^2y^2z^2$, the Alexander dual $I^{[\mathbf{b}]}$ is
well defined.
The exponent vectors of the minimal generators of $I$ are
\[
\mathbf{a}_1=(2,1,0),\qquad
\mathbf{a}_2=(1,0,2),\qquad
\mathbf{a}_3=(0,2,1).
\]
Hence, we have
\[
\begin{aligned}
\mathbf{b}\setminus\mathbf{a}_1
&=(2+1-2,2+1-1,0)=(1,2,0),\\
\mathbf{b}\setminus\mathbf{a}_2
&=(2+1-1,0,2+1-2)=(2,0,1),\\
\mathbf{b}\setminus\mathbf{a}_3
&=(0,2+1-2,2+1-1)=(0,1,2).
\end{aligned}
\]
It follows that
\[
\begin{aligned}
I^{[\mathbf{b}]}
&=\mathfrak{m}^{(1,2,0)}
  \cap\mathfrak{m}^{(2,0,1)}
  \cap\mathfrak{m}^{(0,1,2)}\\
&=\langle x,y^2\rangle\cap\langle x^2,z\rangle\cap\langle y,z^2\rangle.
\end{aligned}
\]
Computing this intersection gives
$I^{[\mathbf{b}]}
=\langle x^2y,xyz,xz^2,y^2z\rangle$. Thus,
$\mu(I^{[\mathbf{b}]})=4.$
By Theorem~\ref{thm_Alex_dual}, we obtain
${\rm i}\omega(I)=\mu(I^{[\mathbf{b}]})=4.$
Moreover, the minimal generators of $I^{[\mathbf{b}]}$ are
\[
x^2y,\quad xyz,\quad xz^2,\quad y^2z.
\]
Applying the operation $\mathbf{b}\setminus\mathbf{a}$ to their
exponent vectors gives, respectively,
\[
(1,2,0),\quad(2,2,2),\quad(2,0,1),\quad(0,1,2).
\]
Therefore, the unique minimal irreducible decomposition of $I$ is given by 
\[
I=\langle x,y^2\rangle
\cap\langle x^2,y^2,z^2\rangle
\cap\langle x^2,z\rangle
\cap\langle y,z^2\rangle.
\]
In particular, this decomposition has four irreducible components,
in agreement with
\[
{\rm i}\omega(I)=\mu(I^{[\mathbf{b}]} )=4.
\]
\end{exm}

From Theorem \ref{thm_witness_irreducible} we can produce immediately the following formula for the omega invariant of a monomial ideal, similar to Corollary \ref{cor_omega_any_monomial} obtained previously.

\begin{cor}\label{cor_omega_ADual} Let $I\subset R=\mathbb K[x_1,\ldots,x_k]$ be a monomial ideal. Let $I^{[{\bf b}]}$ be the Alexander dual of $I$ w.r.t. some exponent vector ${\bf b}$. For $u\in G(I^{[{\bf b}]})$, denote $S_u:={\rm supp}(u)$, $\displaystyle\bar{u}:=\prod_{i\in S_u}x_i$, and for any $\ell\in[k]\setminus S_u$, $b_{\ell,u}=\max\{d_{\ell}(f)\mid f\in G(I)\}$. Also, let $\displaystyle M_u:=\frac{u}{\bar{u}}\cdot\prod_{\ell\in[k]\setminus S_u}x_{\ell}^{b_{\ell,u}}$.

Then,
$$\omega(I)=\left|\left\{(I:M_u)\mid u\in  G(I^{[{\bf b}]})\right\}\right|.$$
\end{cor}
\begin{proof} Let $u=x_{i_1}^{a_1}\cdots x_{i_m}^{a_m}$ be any minimal generator of $I^{[{\bf b}]}$, the Alexander dual of $I$. By Theorem \ref{thm_Alex_dual}, $\langle x_{i_1}^{a_1},\ldots,x_{i_m}^{a_m}\rangle$ is a $\langle x_{i_1},\ldots,x_{i_m}\rangle$-primary component in the minimal irreducible decomposition of $I$. By Theorem \ref{thm_witness_irreducible}, $M_u$ satisfies $(I:M_u)=\langle x_{i_1},\ldots,x_{i_m}\rangle$. But, the cardinality of the set $\{(I:M_u)\mid u\in G(I^{[{\bf b}]})\}$ ensures that it counts only once a prime associated to two different irreducible components, hence the claimed formula.
\end{proof}
\subsubsection{Distinguishing generators of different supports.}\label{sec_diff_supp} From Theorem \ref{thm_Alex_dual}, it is clear that minimal generators of the Alexander dual with the same support will produce irreducible components with the same associated prime ideal. So, in order to obtain $\omega(I)$, rather than ${\rm i}\omega(I)$, one needs to count only those elements of $G(I^{[{\bf b}]})$ that have distinct support. Of course, one can do this by going by hand or with a computer through every element of $G(I^{[{\bf b}]})$ and save those that have different support. Theoretically, we can do this in the following way.

First, we define {\em the $\mu$-support} of a monomial ideal $I\subset R=\mathbb K[x_1,\ldots,x_k]$ as the number of minimal monomial generators of $I$ with different supports; i.e.,

$$\mu{\rm supp}(I)=|\{{\rm supp}(f) \mid f\in G(I)\}|.$$ With this notation, if $I^{[{\bf b}]}$ is the Alexander dual of $I$ w.r.t. some exponent vector ${\bf b}$ then,
$$\omega(I)=\mu{\rm supp}(I^{[{\bf b}]}).$$

In the symbolic computations below, we will also use the classical concept of {\em initial degree}, $\alpha(J)$, of a homogeneous ideal $J$; i.e., $$\alpha(J):=\min\{d\mid J_d\neq 0\} = \min\{\deg(f)\mid f\in J\setminus\{0\}\}.$$

Next, suppose $G(I)=\{f_1,\ldots,f_s\}$; obviously $s=|G(I)|=\mu(I)$, the minimum number of generators of $I$. Let $S:=R[t_1,\ldots,t_s]$ and consider the ideal $$I(t_1,\ldots,t_s):=\langle t_1f_1,\ldots,t_sf_s\rangle\subset S.$$

\begin{prop}\label{prop_mu_support} Let $I\subset R=\mathbb K[x_1,\ldots,x_k]$ be a monomial ideal. Then,

$$\mu{\rm supp}(I)=\mu(I)-\dim_{\mathbb K}{\rm Tor}_2\left(\frac{S}{\sqrt{I(t_1,\ldots,t_{\mu(I)})}},\mathbb K\right)_{\alpha(\sqrt{I})+2}.$$    
\end{prop}
\begin{proof} With the above notations, we have $$\sqrt{I(t_1,\ldots,t_s)}=\left\langle \underbrace{t_1\prod_{i\in{\rm supp}(f_1)}x_i}_{g_1},\ldots,\underbrace{t_s\prod_{i\in{\rm supp}(f_s)}x_i}_{g_s}\right\rangle.$$

Obviously, $\mu(I)=\mu(\sqrt{I(t_1,\ldots,t_s)})$ and $$\alpha(\sqrt{I(t_1,\ldots,t_s)})=1+\min\{|{\rm supp}(f)|\mid f\in G(I)\}=1+\alpha(\sqrt{I}).$$

It is clear that for $i< j$, we have ${\rm supp}(f_i)={\rm supp}(f_j)$ if and only if $t_jg_i-t_ig_j=0$, leading to the linear syzygy on $g_1,\ldots,g_s$, $$(0,\ldots,0,t_j,0,\ldots,0,-t_i,0,\ldots,0).$$ 

If ${\rm supp}(f_{i_1})=\cdots={\rm supp}(f_{i_w}),$ for some $w$, then, for $j=2,\ldots, w$, the $w-1$ many linear syzygies $(t_{i_j},0,\ldots,0,-t_{i_1},0,\ldots,0)$ form a basis for the space of all the linear syzygies on $g_{i_1},\ldots,g_{i_w}$.

Since there are no other linear syzygies among $g_1,\ldots,g_s$ the result comes immediately as expressing the dimension of linear syzygies of $\sqrt{I(t_1,\ldots,t_s)}$ using ${\rm Tor}$.
\end{proof}

\subsubsection{The {\rm v}-number of monomial ideals} The v-number of an ideal was introduced in \cite[Definition 4.1]{CSTVV}. The {\em {\rm v}-number} of a homogeneous ideal $I\subset R=\mathbb K[x_1,\ldots,x_k]$ is

$${\rm v}(I):=\left\{
                          \begin{array}{ll}
                            \min\{d\geq 1\mid \mbox{ there exists } a\in R_d \mbox{ and } P\in {\rm Ass}(R/I) \mbox{ with } (I:a)=P, & \hbox{if $I\subsetneq \frak m$,} \\
                            0, & \hbox{if $I= \frak m$,}
                          \end{array}
                        \right.$$ where $\frak m:=\langle x_1,\ldots,x_k\rangle$. Also by convention, if $P\neq\frak m$ is a prime ideal, ${\rm v}(P)=1$.

The importance of the v-number is reflected in \cite[Section 4.1]{CSTVV}: in certain sufficiently general, yet convenient conditions, it is the index where the minimum distance function on $I$ stabilizes at value 1 (see \cite[Proposition 4.6]{CSTVV}), and also influences the shape of the graded free resolution of $I$ (see \cite[Corollary 4.4, Theorem 4.10, Corollary 4.15]{CSTVV}).

In a snap-shot, the v-number of a homogeneous ideal is the minimum degree of a witness. So from Corollary \ref{cor_omega_ADual} we have the following corollary:

\begin{cor}\label{cor_v_number} Let $I\subset R=\mathbb K[x_1,\ldots,x_k]$ be a monomial ideal. Let $I^{[{\bf b}]}$ be the Alexander dual of $I$ w.r.t. some exponent vector ${\bf b}$. For $u\in G(I^{[{\bf b}]})$, let $S_u:={\rm supp}(u)$, and for any $\ell\in[k]\setminus S_u$, let $b_{\ell,u}:=\max\{d_{\ell}(f)\mid f\in G(I)\}$. Then, 
$${\rm v}(I)\leq\min_{u\in G(I)}\left\{\deg(u)-|S_u|+\sum_{\ell\in[k]\setminus S_u}b_{\ell,u}\right\}.$$
\end{cor}

It is not hard to construct examples when we have equality in the above result. Because Corollary \ref{cor_omega_ADual} and consequently, Corollary \ref{cor_v_number}, uses the minimal irreducible decomposition of $I$ rather than a minimal primary decomposition, we should expect examples where the inequality is strict. This is indeed the case, as we can see in the next example, which is \cite[Example 4.3]{CSTVV}.

\begin{exm}\label{exm_v_number} Consider the following ideal in $\mathbb Q[x_1,x_2,x_3,x_4]$ given by the following minimal primary decomposition

$$I=\langle x_2^{10},x_3^9,x_4^4,x_2x_3x_4^3\rangle\cap\langle x_1^4,x_3^4,x_4^3,x_1x_3x_4^2\rangle \cap\langle x_1^4, x_2^5,x_4^3\rangle\cap\langle x_1^3, x_2^5,x_3^{10}\rangle.$$

With \cite{GrSt}, we have
\begin{align*}
G(I)=\{& x_1^4x_2^{10},x_2^{10}x_3^4,x_1^4x_3^9,x_2^5x_3^9,x_1x_2^{10}x_3x_4^2,x_2^{10}x_4^3,\\
& x_1^3x_2x_3x_4^3,x_2^5x_3x_4^3,x_1^3x_3^9x_4^3,x_3^{10}x_4^3,x_1^3x_4^4,x_2^5x_4^4\},
\end{align*}
   
and so, the maximum degrees of each variable in the minimal generators of $I$ are $$d_{x_1}=4, d_{x_2}=10, d_{x_3}=9,d_{x_4}=4;$$ we need these when we construct the witnesses.

The first two primary components of $I$ are not irreducible ideals, each decomposing as intersection of three irreducible ideals that do not overlap. So ${\rm i}\omega(I)=8$. In the table below we list each irreducible component with its corresponding witness, $M_u$, from Corollary \ref{cor_omega_ADual}.

\medskip

\begin{center}
\begin{tabular}{c|c|c}

irreducible component corresponding to $u$ & $M_u$ & $\deg(M_u)$\\
\hline

$\langle x_2,x_3^9,x_4^4\rangle$ & $(x_2^0x_3^8x_4^3)x_1^4$ & $15$ \\

\hline

$\langle x_2^{10},x_3,x_4^4\rangle$ & $(x_2^9x_3^0x_4^3)x_1^4$ & $16$\\

\hline

$\langle x_2^{10},x_3^9,x_4^3\rangle$ & $(x_2^9x_3^8x_4^2)x_1^4$ & $23$\\

\hline

$\langle x_1,x_3^4,x_4^3\rangle$ & $(x_1^0x_3^3x_4^2)x_2^{10}$ & $15$\\

\hline

$\langle x_1^4,x_3,x_4^3\rangle$ & $(x_1^3x_3^0x_4^2)x_2^{10}$ & $15$\\

\hline

$\langle x_1^4,x_3^4,x_4^2\rangle$ & $(x_1^3x_3^3x_4^1)x_2^{10}$ & $17$\\

\hline

$\langle x_1^4,x_2^5,x_4^3\rangle$ & $(x_1^3x_2^4x_4^2)x_3^{9}$ & $18$\\

\hline

$\langle x_1^3,x_2^5,x_3^{10}\rangle$ & $(x_1^2x_2^4x_3^9)x_4^{4}$ & $19$
\end{tabular}
\end{center}

The minimum value showing up in the last column of the table above is $15$. But $(I:x_2^9x_4^3)=\langle x_2,x_3,x_4\rangle$, so ${\rm v}(I)\leq 12$. In fact, as \cite[Example 4.3]{CSTVV} mentions, ${\rm v}(I)=12$.

\end{exm}

\medskip

\subsection{The maximal standard monomials (msm) method}\label{sec_msm}

Let $I\subset R:=\mathbb K[x_1,\ldots,x_k]$ be a monomial ideal. A monomial $M\in R$ is called {\em maximum standard monomial of $I$} if $M\notin I$ and $x_1M,\ldots,x_kM\in I$. This monomial is often referred to in literature as a ``corner element" or a ``socle element".

Let $J$ be an ideal of $R$ such that ${\rm Ass}(J)=\{\langle x_1,\ldots,x_k\rangle\}$; i.e., the algebra $R/J$ is {\em Artinian}. Then, {\em the socle of $R/J$} is the finite dimensional $\mathbb K$-vector space $${\rm Soc}(R/J):=\{\overline{f}\in R/J\,|\, x_1f,\ldots,x_kf\in J\}.$$

We have the following result concerning the irreducible omega invariant of a monomial ideal.

\begin{thm}\label{thm_msm} (\cite[Exercise 5.8]{MiSt}, \cite[Proposition 3 in Section 2.2]{Ro}) Let $I$ be a monomial ideal in $R=\mathbb K[x_1,\ldots,x_k]$, and let $t$ be an integer strictly larger than the degree of any minimal generator of $I$. Then
$${\rm i}\omega(I)=\dim_{\mathbb K}{\rm Soc}(R/(I+\langle x_1^t,\ldots,x_k^t\rangle).$$    
\end{thm}
\begin{proof} The proof is given by the Slice Algorithm, and we use this occasion to present it here. 

If ${\bf a}=(a_1,\ldots,a_k)$ is an exponent vector, define ${\bf a}(t)=(b_1,\ldots,b_k)$ with $$b_i:=\left\{
                          \begin{array}{ll}
                            a_i+1, & \hbox{if $a_i+1<t$,} \\
                            0, & \hbox{otherwise.}
                          \end{array}
                        \right.$$

Then we have ${\bf x}^{\bf a}$ is a maximum standard monomial of $I+\langle x_1^t,\ldots,x_k^t\rangle$ if and only if $\frak m^{{\bf a}(t)}$ is an irredundant irreducible component of $I$.

But, by the definition, the maximum standard monomials form a basis for the socle.
\end{proof}

\medskip 


\begin{exm}
Let $R = \mathbb K[x_1, x_2, x_3, x_4]$ and consider the monomial ideal
\[
I = \langle x_1 x_2, x_1 x_3, x_1 x_4, x_2 x_3, x_2 x_4, x_3 x_4, x_1^4, x_2^4, x_3^4, x_4^4\rangle.
\]
The degree of every minimal generator of $I$ is at most $4$, so we may take $t = 5$ in Theorem~\ref{thm_msm}. Set
\[
J := I + \langle x_1^5, x_2^5, x_3^5, x_4^5 \rangle.
\]
Since $x_i^4 \in I$ for each $i$, we have $x_i^5 \in I$ as well, and hence $J = I$.
The monomials not contained in $I$ are
\[
1, \quad x_i, \quad x_i^2, \quad x_i^3 \qquad (1 \le i \le 4).
\]
Indeed, any monomial involving at least two distinct variables belongs to $I$, while $x_i^4 \in I$. Among these standard monomials, the maximum standard monomials are precisely
\[
x_1^3, \quad x_2^3, \quad x_3^3, \quad x_4^3.
\]
Consequently,
\[
\operatorname{Soc}(R/J) =
\bigoplus_{x^\alpha\in\operatorname{MSM}(J)}
\mathbb{K}\,\overline{x^\alpha}=
\mathbb K \overline{x_1^3} \oplus \mathbb K \overline{x_2^3} \oplus \mathbb K \overline{x_3^3} \oplus \mathbb K \overline{x_4^3},
\]
and therefore $\dim_{\mathbb K} \operatorname{Soc}(R/J) = 4$. By Theorem~\ref{thm_msm}, it follows that ${\rm i}\omega(I) = 4$.
Now, we apply the Slice Algorithm. The exponent vectors of the maximum standard monomials are
\[
(3,0,0,0), \quad (0,3,0,0), \quad (0,0,3,0), \quad (0,0,0,3).
\]
Since $t = 5$, applying the map $\mathbf{a} \mapsto \mathbf{a}(5)$ yields that 
\[
\begin{aligned}
(3,0,0,0)(5) &= (4,1,1,1), \\
(0,3,0,0)(5) &= (1,4,1,1), \\
(0,0,3,0)(5) &= (1,1,4,1), \\
(0,0,0,3)(5) &= (1,1,1,4).
\end{aligned}
\]
Thus, the corresponding irreducible components are
\[
\begin{aligned}
\mathfrak m^{(4,1,1,1)} &= \langle x_1^4, x_2, x_3, x_4\rangle, \\
\mathfrak m^{(1,4,1,1)} &= \langle x_1, x_2^4, x_3, x_4\rangle, \\
\mathfrak m^{(1,1,4,1)} &= \langle x_1, x_2, x_3^4, x_4\rangle, \\
\mathfrak m^{(1,1,1,4)} &= \langle x_1, x_2, x_3, x_4^4\rangle.
\end{aligned}
\]
Hence, the algorithm gives the minimal irreducible decomposition as follows
\[
I = \langle x_1^4, x_2, x_3, x_4\rangle \cap \langle x_1, x_2^4, x_3, x_4\rangle \cap \langle x_1, x_2, x_3^4, x_4\rangle \cap \langle x_1, x_2, x_3, x_4^4\rangle.
\]
Since these four components are irredundant, we confirm that ${\rm i}\omega(I) = 4$.
\end{exm}
\medskip

\subsection{Omega invariants of Veronese type ideals.}\label{sec_Veronese} If above we looked at monomial ideals in general, below we will do a qualitative analysis of the class of Veronese type ideals.

Let $R=\mathbb K[x_1,\ldots,x_k]$ be the graded ring of polynomials with coefficients in a field, with standard grading given by the degree. Denote $\frak m:=\langle x_1,\ldots,x_k\rangle$, the irrelevant maximal ideal. 

Let $\Sigma=(\ell_1,\ldots,\ell_n)\in R$ be linear forms, some possibly proportional. The {\em rank} of $\Sigma$ is ${\rm rk}(\Sigma):={\rm ht}(\langle \ell_1,\ldots,\ell_n\rangle)$. Let $a\geq 0$ be an integer. With the convention that $I_0(\Sigma)=R$ and $I_a(\Sigma)=0$, if $a\geq n+1$, the ideal $$I_a(\Sigma):=\langle\{\ell_{i_1}\cdots\ell_{i_a}|1\leq i_1<\cdots<i_a\leq n\}\rangle\subset R$$ is called {\em the ideal generated by $a$-fold products of linear forms of $\Sigma$}.

Maybe the most important result concerning these ideals is \cite[Theorem 2.2]{BuToXi} (or \cite[Theorem 3.35]{To}) which says that for $1\leq a\leq n$ and for any $\Sigma$, the ideals $I_a(\Sigma)$ have linear free resolution and that they have a primary decomposition as follows: let $\Gamma(\Sigma):=\{\langle A \rangle | A\subseteq\Sigma\}$ be the set of all prime ideals generated by linear forms on $\Sigma$. Then, for $1\leq a\leq n$, a primary decomposition of $I_a(\Sigma)$ is
$$I_a(\Sigma)=\bigcap_{\frak p\in\Gamma(\Sigma)}\frak p^{a-n+\nu_{\Sigma}(\frak p)},$$ where $\nu_{\Sigma}(\frak p)$ is the number of linear forms of $\Sigma$ that belong to $\frak p$, counted with multiplicity.

\medskip

Veronese type ideals are a particular case of ideals generated by fold products of linear forms. Suppose $\Sigma=(\underbrace{x_1,\ldots,x_1}_{m_1},\ldots,\underbrace{x_k,\ldots,x_k}_{m_k})\subset R=\mathbb K[x_1,\ldots,x_k]$, and let $1\leq a\leq n=m_1+\cdots+m_k$. Then, $I_a(\Sigma)$ is generated by all monomials $x_1^{i_1}\cdots x_k^{i_k}$ with $i_1+\cdots+i_k=a$ and $i_j\leq m_j$ for all $j=1,\ldots,k$. Such an ideal is called {\em a Veronese type ideal}.

\begin{thm}\label{thm_Veronese} Let $I_a(\Sigma)$ be a Veronese type ideal as above. Then, $$\omega(I_a(\Sigma))=\sum_{t=1}^k|\{(i_1,\ldots,i_t)|1\leq i_1<\cdots<i_t\leq k \mbox{ and } a+m_{i_1}+\cdots+m_{i_t}>n\}|.$$
\end{thm}
\begin{proof} We apply \cite[Theorem 2.2]{BuToXi} that we mentioned at the beginning of this subsection. A prime ideal $\frak p$ in $\Gamma(\Sigma)$ is of the form $\frak p=\langle x_{i_1},\ldots,x_{i_t}\rangle$, with $1\leq i_1<\cdots<i_t\leq k$. In this case, clearly, $$\nu_{\Sigma}(\frak p)=m_{i_1}+\cdots+m_{i_t}.$$

If $s<n$, then $a-n+m_{i_1}+\cdots+m_{i_s}<a-n+m_{i_1}+\cdots+m_{i_t}$, and so, if $\frak p_1=\langle x_{i_1},\ldots,x_{i_s}\rangle$ is an associated prime of $I_a(\Sigma)$, then so is $\frak p_2=\langle x_{i_1},\ldots,x_{i_t}\rangle$, despite that $\frak p_1\subsetneq\frak p_2$.\footnote{This reiterates the note following \cite[Corollary 2.4]{ToXi} that for Veronese type ideals, the primary decomposition in \cite[Theorem 2.2]{BuToXi} (first obtained in \cite{ToXi} for $\Sigma$ with generic support) is a minimal primary decomposition.} 

So the claim is shown.
\end{proof}

\begin{lem}\label{lem_Veronese} Let $I_a(\Sigma)$ be a Veronese type ideal, and let ${\bf m}:=(m_1,m_2,\ldots,m_k)$. Then the Alexander dual w.r.t. ${\bf m}$ is $$I_a(\Sigma)^{[{\bf m}]}=I_{n-a+1}(\Sigma).$$
\end{lem}
\begin{proof} Let $x_1^{i_1}\cdots x_k^{i_k}\in I_a(\Sigma)$, so $i_1+\cdots+i_k=a$, and $0\leq i_j\leq m_j$, for all $j=1,\ldots,k$. 

Since $(i_1,\ldots,i_k)\preceq (m_1,\ldots,m_k)$, it follows that $x_1^{i_1}\cdots x_k^{i_k}$ divides $x_1^{m_1}\cdots x_k^{m_k}$, so by \cite[Proposition 5.23]{MiSt} we have
$$x_1^{m_1-i_1}\cdots x_k^{m_k-i_k}\not\in I_a(\Sigma)^{[{\bf m}]},$$
where $\displaystyle \sum_{j=1}^k (m_j-i_j)=n-a$. Since, any $k$-tuple $(t_1,\ldots,t_k)$ with $t_1+\cdots+t_k=n-a$ and $0\leq t_j\leq m_j$ for all $j=1,\ldots,k$ can be written as $(m_1-i_1,\ldots,m_k-i_k)$ with $i_1+\cdots+i_k=a$ and $0\leq i_j\leq m_j$ for all $j=1,\ldots,k$, we obtain that $I_a(\Sigma)^{[{\bf m}]}$ does not contain any $(n-a)$-fold products of $\Sigma$. 

On the other hand, for any $j=1,\ldots,k$ with $i_j\neq 0$,
$$\frac{x_1^{i_1}\cdots x_k^{i_k}}{x_j}\not\in I_a(\Sigma)$$
because it is an $(a-1)$-fold product, so \cite[Proposition 5.23]{MiSt} implies that
\[
    x_j(x_1^{m_1-i_i}\cdots x_k^{m_k-i_k}) \in I_a(\Sigma)^{[{\bf m}]},
\]
where $x_j(x_1^{m_1-i_i}\cdots x_k^{m_k-i_k})$ is an $(n-a+1)$-fold product of $\Sigma$. But any $(n-a+1)$-fold product is an $(n-a)$-fold product times a variable. So $I_a(\Sigma)^{[{\bf m}]}$ contains all $(n-a+1)$-fold products, and as a monomial ideal it must be generated by such. Hence $I_a(\Sigma)^{[{\bf m}]}=I_{n-a+1}(\Sigma)$.
\end{proof}


\begin{exm}
Let  $R=\mathbb K[x_1,x_2,x_3,x_4]$ and consider
$$\Sigma=(x_1,x_1,x_2,x_2,x_2,x_3,x_4,x_4).$$
Thus
\[
(m_1,m_2,m_3,m_4)=(2,3,1,2),
\qquad n=m_1+m_2+m_3+m_4=8.
\]
Let $a=5$. Then $I_5(\Sigma)$ is the Veronese type ideal
\[
I_5(\Sigma)
=
\left\langle
x_1^{i_1}x_2^{i_2}x_3^{i_3}x_4^{i_4}
:
\begin{array}{l}
i_1+i_2+i_3+i_4=5,\\
i_1\leq 2,\ i_2\leq 3,\ i_3\leq 1,\ i_4\leq 2
\end{array}
\right\rangle.
\]
It follows from  Theorem~\ref{thm_Veronese} that 
\[
\omega(I_5(\Sigma))
=
\sum_{t=1}^4
\left|
\left\{
S\subseteq\{1,2,3,4\}:
|S|=t,\quad
5+\sum_{i\in S}m_i>8
\right\}
\right|.
\]
Equivalently, we need to count the subsets $S$ such that
\[
\sum_{i\in S}m_i>3.
\]
For subsets of cardinality one, none satisfies this condition, since
\[
m_1=2,\qquad m_2=3,\qquad m_3=1,\qquad m_4=2.
\]
Hence the contribution for $t=1$ is $0$.
For subsets of cardinality two, we have
\[
\begin{array}{c|c}
S & \displaystyle\sum_{i\in S}m_i \\ \hline
\{1,2\} & 5\\
\{1,3\} & 3\\
\{1,4\} & 4\\
\{2,3\} & 4\\
\{2,4\} & 5\\
\{3,4\} & 3
\end{array}
\]
Thus, the subsets satisfying the condition are
\[
\{1,2\},\quad
\{1,4\},\quad
\{2,3\},\quad
\{2,4\},
\]
and the contribution for $t=2$ is $4$.
Every subset of cardinality three satisfies the condition, since
\[
\begin{aligned}
m_1+m_2+m_3&=6, &
m_1+m_2+m_4&=7,\\
m_1+m_3+m_4&=5, &
m_2+m_3+m_4&=6.
\end{aligned}
\]
Hence the contribution for $t=3$ is $4$.
Finally,
\[
m_1+m_2+m_3+m_4=8>3,
\]
so the unique subset of cardinality four also contributes $1$.
Consequently, we get 
\[
\omega(I_5(\Sigma))=0+4+4+1=9.
\]
More explicitly, the nine associated primes are
\[
\begin{aligned}
\operatorname{Ass}(I_5(\Sigma))
=\big\{&
\langle x_1,x_2\rangle,\ \langle x_1,x_4\rangle,\ \langle x_2,x_3\rangle,\ \langle x_2,x_4\rangle,\\
&
\langle x_1,x_2,x_3\rangle,\ \langle x_1,x_2,x_4\rangle,\
\langle x_1,x_3,x_4\rangle,\ \langle x_2,x_3,x_4\rangle,\\
&
\langle x_1,x_2,x_3,x_4\rangle
\big\}.
\end{aligned}
\]
\end{exm}

\medskip

Immediately from Lemma \ref{lem_Veronese} and Theorem \ref{thm_Alex_dual} one has the corollary:

\begin{cor}\label{cor_Veronese} With the above notations one has
$$\omega(I_a(\Sigma))\leq \mu(I_{n-a+1}(\Sigma)).$$
\end{cor}

By \cite[Theorem 2.2]{AbZa}, we have the following complicated formula for the minimum number of generators for the Veronese type ideal $I_{n-a+1}(\Sigma)$: if $1\leq m_i\leq n-a+1$, for all $i=1,\ldots,k$, then

$$\mu(I_{n-a+1}(\Sigma))={{n-a+k}\choose{k-1}}+\sum_{J\subseteq\{1,\ldots,k\}}(-1)^{|J|}{{n-a+k-\sum_{i\in J}(m_i+1)}\choose{k-1}}.$$

\begin{exm}\label{exm_Veronese} Let $\Sigma=(x,x,x,x,y,y,z,z)$, and let $a=4$. With \cite{GrSt} we have

$$\omega(I_4(\Sigma))=3 < 8 =\mu(I_{8-4+1}(\Sigma)).$$
\end{exm}
\medskip

\begin{prop}\label{prop_iomega_veronese} Let $I=\langle x_1,\ldots,x_s\rangle^a$, for some $s\in\{1,\ldots,k\}$ and some $a\geq 1$, an integer. Then,
    $${\rm i}\omega(I)=\sum_{i=0}^s(-1)^{i}{{s}\choose{i}}{{sa-a+s-i(a+1)}\choose{s-1}}.$$
\end{prop}
\begin{proof}
    Let $\Sigma=(\underbrace{x_1,\ldots,x_1}_{a},\ldots,\underbrace{x_s,\ldots,x_s}_{a})$. Then $I=I_a(\Sigma)$. WLOG we can suppose that $s=k$. Hence, if ${\bf m}:=(\underbrace{a,\ldots,a,}_{s})$ by Lemma \ref{lem_Veronese}, $I^{[{\bf m}]}=I_{sa-a+1}(\Sigma)$. Next, we apply the formula from \cite[Theorem 2.2]{AbZa} mentioned above. 
\end{proof}

\section{Witnesses and omega invariants of powers of monomial ideals}
\label{Matrix_Method}

Let $I$ with $G(I)=\langle f_1,\ldots,f_s\rangle$ be a monomial ideal in $R=\mathbb K[x_1,\ldots,x_k]$, where
\[
f_j=x_1^{a_{1j}}\cdots x_k^{a_{kj}},
\qquad
j=1,\ldots,s.
\]
We denote by $A=(a_{ij})\in\mathbb N_0^{k\times s}$ the exponent matrix of $I$, defined as follows:
\[
A=
\begin{pmatrix}
a_{11} & \cdots & a_{1s}\\
\vdots & & \vdots\\
a_{k1} & \cdots & a_{ks}
\end{pmatrix}.
\]
For a vector $\lambda=(\lambda_1,\ldots,\lambda_s)^T\in\mathbb N_0^s$ with
$|\lambda|:=\lambda_1+\cdots+\lambda_s=n,$ the monomial
$f^\lambda:=f_1^{\lambda_1}\cdots f_s^{\lambda_s}$ belongs to $I^n$. Its exponent vector is
\[
\operatorname{exp}(f^\lambda)
=
\begin{pmatrix}
\displaystyle\sum_{j=1}^s a_{1j}\lambda_j\\
\vdots\\
\displaystyle\sum_{j=1}^s a_{kj}\lambda_j
\end{pmatrix}
=
A\lambda.
\]
Thus, every product of $n$  minimal generators of $I$ is represented by an
integer vector $\lambda$ satisfying $|\lambda|=n$.
This observation allows us to study associated primes of $I^n$
without explicitly computing $G(I^n)$.

We note here that in parallel and independent to our work, \cite{JMi} uses the same exponent matrix to exhibit a primary decomposition of $I$. Inspired by the quotient method (see Section \ref{sec_quotientl}), their result exhibits a primary decomposition in a non-recursive way, and it is based on the fact that if a prime ideal is associated to a monomial ideal, then it must be generated by the subsets of variables (see the beginning of Section  \ref{sec_monomial}).     

\subsection{Witness systems for powers}
\label{sec_power_witnesses}

In what follows, we will generalize Definition \ref{defn_witness_exp}, Theorem \ref{thm_witness_any_monomial}, Corollary \ref{cor_omega_any_monomial}, and also Example \ref{exm_assoc} to any power $n\geq 1$ of the monomial ideal $I$.

Let $S=\{i_1,\ldots,i_m\}\subseteq [k],$ and put $P_S=\langle x_{i_1},\ldots,x_{i_m}\rangle.$
Suppose that, for each $u=1,\ldots,m$, we choose a vector $\lambda^{(u)}\in\mathbb N_0^s$ with $|\lambda^{(u)}|=n.$
Write $b^{(u)}:=A\lambda^{(u)}.$ Thus,  $b^{(u)}$ is the exponent vector of the monomial
$f^{\lambda^{(u)}}$, and the $i$-th entry of $b^{(u)}$ is the degree of the variable $x_i$ in $f^{\lambda^{(u)}}$, i.e., with our previous notations,

$$b^{(u)}_i=d_i(f^{\lambda^{(u)}}).$$

We say that the selected vectors form an
\emph{$S$-dominant witness system} if
\begin{equation}
b^{(u)}_{i_u}>b^{(u)}_{i_v} \qquad \text{for all } v\in\{1,\ldots,m\}\setminus\{u\}.
\label{eq:power_dominance}
\end{equation}

The distinguished exponents are $t_u:=b^{(u)}_{i_u}$, where  $u=1,\ldots,m.$
For the construction of a  monomial witness, we use the
following stronger compatibility condition.

\begin{defn}
\label{def:compatible_witness}
An $S$-dominant witness system $\lambda^{(1)},\ldots,\lambda^{(m)}$
is called \emph{compatible} if, in addition to \eqref{eq:power_dominance}, it satisfies
\begin{equation}
b^{(u)}_{i_v} < b^{(v)}_{i_v}
\qquad \text{for all }u,v\in\{1,\ldots,m\}\text{ with }u\neq v.
\label{eq:cross_dominance}
\end{equation}
\end{defn}
The first condition identifies the distinguished variable in each
selected monomial, while the second condition ensures the compatibility
of these distinguished exponents in the construction of a monomial witness.

The remaining condition can be interpreted as an escape problem.
More precisely, we say that a monomial $f^\lambda\in I^n$ \emph{escapes
the proposed witness} if
\[
(A\lambda)_{i_u}<t_u
\qquad
\text{for every }u=1,\ldots,m.
\]
Thus, an escaping monomial is one whose exponent at each distinguished
variable $x_{i_u}$ is strictly smaller than the corresponding
distinguished exponent $t_u$.

Consequently, we consider the following integer feasibility system:
\begin{equation}
\begin{cases}
\lambda_1+\cdots+\lambda_s=n,\\
\lambda_j\geq0,\qquad j=1,\ldots,s,\\
(A\lambda)_{i_u}<t_u,\qquad u=1,\ldots,m.
\end{cases}
\tag{$\mathcal E_{S,n}$}
\label{eq:escape_system}
\end{equation}

We call \eqref{eq:escape_system} the
\emph{$S$-escape system}.


The preceding discussion reduces the problem of detecting associated primes
of powers of monomial ideals to the existence of compatible witness systems
together with the nonexistence of escaping products. Here is a simple example.

The preceding discussion provides a criterion for constructing a
monomial witness for $P_S$ at the power $I^n$. Namely, suppose that
$\lambda^{(1)},\ldots,\lambda^{(m)}$ form a compatible $S$-dominant
witness system, and let
\[
t_u=b^{(u)}_{i_u},\qquad u=1,\ldots,m.
\]
The associated witness monomial has the form
\[
M=x_{i_1}^{t_1-1}\cdots x_{i_m}^{t_m-1}M',
\]
where $M'$ is a monomial determined by the construction in
Theorem~\ref{thm_witness_any_monomial}.
To verify that this monomial is a witness, it is necessary to rule out
a product of $n$ minimal generators of $I$ whose exponent vector
satisfies
\[
(A\lambda)_{i_u}<t_u
\qquad\text{for all }u=1,\ldots,m.
\]
We call such a product an \emph{escaping product} for the proposed
witness. Equivalently, an escaping product is a monomial
$f^\lambda\in I^n$, with $\lambda\in\mathbb N_0^s$ and
$|\lambda|=n$, whose exponent at every distinguished variable
$x_{i_u}$ is strictly smaller than the corresponding threshold $t_u$.
Thus, the existence of an escaping product is exactly the existence of
a solution to the integer feasibility system
\[
\begin{cases}
\lambda_1+\cdots+\lambda_s=n,\\
\lambda_j\geq0,\qquad j=1,\ldots,s,\\
(A\lambda)_{i_u}<t_u,\qquad u=1,\ldots,m.
\end{cases}
\tag{$\mathcal E_{S,n}$}
\label{eq:escape_system1}
\]
Consequently, the nonexistence of a solution to
\eqref{eq:escape_system1} guarantees that no product of $n$ minimal
generators of $I$ can simultaneously have exponent less than $t_u$ at
every distinguished variable $x_{i_u}$. This is precisely the
condition needed in the construction of the monomial witness.


\begin{exm}
    Let
\[
I=\langle x_1,x_2,x_3\rangle,\qquad n=2,\qquad S=\{1,2\}.
\]
We have
\[
A=I_3,\qquad
t_1=b^{(1)}_1=2,\qquad t_2=b^{(2)}_2=2.
\]
Thus $\mathcal E_{S,2}$ is
\[
\begin{cases}
\lambda_1+\lambda_2+\lambda_3=2,\\
\lambda_1,\lambda_2,\lambda_3\geq0,\\
\lambda_1<2,\\
\lambda_2<2.
\end{cases}
\]
But, for example, $\lambda=(1,1,0)$
is a solution. Indeed,
$A\lambda=(1,1,0),$
and therefore
\[
(A\lambda)_1=1<2=t_1,\qquad
(A\lambda)_2=1<2=t_2.
\]
The corresponding monomial is
$f^\lambda=x_1x_2,$
which is an element of $I^2$ and \textbf{escapes the proposed witness}, because its $x_1$- and $x_2$-degrees are both strictly smaller than the distinguished exponents $2,2$.
So the conclusion is:
\[
\mathcal E_{S,2}\text{ has a solution, hence }
P_S\notin\operatorname{Ass}(R/I^2)
\text{ by the criterion.}
\]

\end{exm}

\medskip

Generalizing Theorem \ref{thm_witness_any_monomial}, the following theorem
makes this observation precise by providing a matrix-based criterion for
membership in $\operatorname{Ass}(R/I^n)$.


\begin{thm}{\rm (Matrix witness criterion for powers)}  \label{thm:matrix_power_witness}
Let $I \subset R=\mathbb K[x_1,\ldots,x_k]$ be a monomial ideal minimally generated by $G(I)=\{f_1,\ldots,f_s\}$, and with exponent matrix $A$. Fix $n\geq1$ and let
$S=\{i_1,\ldots,i_m\}\subseteq[k].$  Suppose there exist vectors
 $\lambda^{(1)},\ldots,\lambda^{(m)} \in\mathbb N_0^s$ with $|\lambda^{(u)}|=n$ for all $1\leq u \leq m$
 such that the associated monomials form a compatible $S$-dominant witness system. Set
 $t_u=(A\lambda^{(u)})_{i_u}.$ Assume, in addition, that $t_u\geq 1$ for all $1\leq u \leq m$.
 If the integer system $\mathcal E_{S,n}$ has no solution, then
 $P_S=\langle x_{i_1},\ldots,x_{i_m}\rangle \in \operatorname{Ass}(R/I^n).$
\end{thm}

\begin{proof}
 For each $u=1,\ldots,m$, let $g_u=f^{\lambda^{(u)}}\in I^n.$ The exponent vector of $g_u$ is $b^{(u)}$.
For each $j\notin S$, set
\[c_j:=\max\left\{(A\lambda)_j:\lambda\in\mathbb N_0^s,\ |\lambda|=n\right\}.\]
This maximum exists because there are only finitely many vectors $\lambda\in\mathbb N_0^s$ satisfying $|\lambda|=n$.
 Define the monomial
\begin{equation}
M
=
\prod_{u=1}^m x_{i_u}^{t_u-1}
\prod_{j\notin S}x_j^{c_j}.
\label{eq:matrix_witness_M}
\end{equation}
The assumption $t_u\geq1$ guarantees that all exponents in \eqref{eq:matrix_witness_M} are nonnegative.

We first prove that $P_S\subseteq(I^n:M).$ Fix $u\in\{1,\ldots,m\}$. 

We claim that $g_u\mid x_{i_u}M.$ With the previous notation that for an arbitrary monomial $h$, $d_i(h)$ denotes the exponent of the variable $x_i$ showing up in $h$, we have $d_{i_u}(g_u)=t_u=d_{i_u}(x_{i_u}M)$.
For $v\neq u$, compatibility gives that $
d_{i_v}(g_u)
=
b^{(u)}_{i_v}
<
b^{(v)}_{i_v}
=
t_v.
$
Since all exponents are integers, we can deduce that $
b^{(u)}_{i_v}
\leq
t_v-1
=
d_{i_v}(M).
$
Finally, for $j\notin S$, we have $
d_j(g_u)
\leq
c_j
=
d_j(M).
$
Hence, $g_u\mid x_{i_u}M.$ Therefore, $x_{i_u}M\in I^n,$
and consequently $x_{i_u}\in(I^n:M).$ But this holds for every $u$, so we obtain
$P_S\subseteq(I^n:M).$

We now prove the reverse inclusion. Let $N$  be an arbitrary monomial such that  $N\notin P_S$.
 Then none of  $x_{i_1},\ldots,x_{i_m}$ divides $N$. Hence, we have $d_{i_u}(N)=0$
 for every $u$, and so we can deduce that
\begin{equation}
d_{i_u}(NM)
=
t_u-1
\qquad
\text{for every }u.
\label{eq:degree_NM}
\end{equation}
Suppose, to the contrary, that $NM\in I^n.$ Since $I^n$ is generated by products of $n$ generators of $I$, there
exists $\lambda\in\mathbb N_0^s$ with $|\lambda|=n$ such that  $f^\lambda\mid NM.$
It follows from \eqref{eq:degree_NM} that
\[
(A\lambda)_{i_u}
\leq
t_u-1
<
t_u
\qquad
\text{for every }u.
\]
Thus,  $\lambda$ is a solution of the escape system $\mathcal E_{S,n}$, contradicting the assumed infeasibility of that
system. Therefore, $NM\notin I^n$ for every monomial $N\notin P_S$. Hence, $(I^n:M)\subseteq P_S.$
Together with $P_S\subseteq(I^n:M)$, this leads to  $(I^n:M)=P_S.$
 We therefore get   $P_S\in\operatorname{Ass}(R/I^n)$, as required.
\end{proof}


The significance of Theorem~\ref{thm:matrix_power_witness} is that the
minimal generating set $G(I^n)$ is never required. A monomial in $I^n$
is represented by an integer vector $\lambda$ satisfying
$|\lambda|=n$, and its exponent vector is obtained by the matrix
multiplication $A\lambda.$ Thus, the problem of detecting an associated prime of $I^n$ is reduced
to an integer feasibility problem involving the exponent matrix of the
original ideal $I$.

The following examples illustrate how Theorem~\ref{thm:matrix_power_witness} can be applied using the exponent matrix and the associated
escape system. In particular, Example~\ref{Example-1} recovers the result established in~\cite[Proposition~3.6]{NKA}.

We first recall the following basic notions and definitions that will be
used in the next example.

Let $G=(V(G),E(G))$ be a finite simple graph. Suppose that $
V(G)=\{x_1,\ldots,x_m\}
$
and let $R=K[x_1,\ldots,x_m]$.
The \textit{edge ideal} of $G$ is the following square-free monomial ideal
\[
I(G)=\left\langle x_ix_j:\{x_i,x_j\}\in E(G)\right\rangle
\subseteq R.
\]
Thus, $I(G)$ is generated by the quadratic monomials corresponding to
the edges of $G$. In other words, for each edge
$\{x_i,x_j\}\in E(G)$, the monomial $x_ix_j$ is a generator of $I(G)$.

A \textit{vertex cover} of $G$ is a subset $C\subseteq V(G)$ such that
every edge of $G$ has at least one endpoint in $C$; equivalently, for
every $\{u,v\}\in E(G)$, we have $u\in C$ or $v\in C$. A vertex cover
$C$ is called \textit{minimal} if no proper subset of $C$ is a vertex
cover.
The \textit{cover ideal} of $G$ is the following square-free monomial ideal
\[
J=J(G)=\bigcap_{\{x_i,x_j\}\in E(G)}(x_i,x_j).
\]
Equivalently, $J(G)$ is generated by the monomials corresponding to the
minimal vertex covers of $G$. Thus, if
$C_1,\ldots,C_r$ are the minimal vertex covers of $G$, then
\[
J(G)=\langle x_{C_1},\ldots,x_{C_r}\rangle,
\qquad
x_{C_j}=\prod_{x_i\in C_j}x_i.
\]

A subset $U\subseteq V(G)$ is called an \textit{independent set} if no
two distinct vertices in $U$ are adjacent; that is,
\[
\{u,v\}\notin E(G)
\qquad\text{for all distinct }u,v\in U.
\]
Equivalently, the subgraph of $G$ induced by $U$ has no edges.


\begin{exm} \label{Example-1}
Suppose that  $G=C_{2n+1}$ denotes  the odd  cycle graph with the vertex set  $V(G)=\{x_1,\ldots,x_{2n+1}\}$  and the following  edge set
\[
E(G)=
\big\{\{x_i,x_{i+1}\}:1\leq i\leq 2n\big\}
\cup
\big\{\{x_{2n+1},x_1\}\big\}.
\]
Let $J$ denote the cover ideal of $G$ and $R=\mathbb K[x_1, \ldots, x_{2n+1}]$.  Also, let
$
G(J)=\{h_1,\ldots,h_r\},
$
where $r=\mu(J)$, and let $A$ be the $(2n+1)\times r$ exponent matrix of $J$.

 Using Theorem~\ref{thm:matrix_power_witness}, we show that
  $\langle x_1,\ldots,x_{2n+1}\rangle\in\operatorname{Ass}(R/J^s)$ for all $s\geq 2$.  Consider the independent set
 $U:=\{x_1,x_3,\ldots,x_{2n-1}\}.$  Put
\[
C:=V(G)\setminus U
=\{x_2,x_4,\ldots,x_{2n},x_{2n+1}\}.
\]
Thus, $|U|=n$ and $|C|=n+1$. For each $i\in\{1,\ldots,2n+1\}$, the graph $G\setminus\{x_i\}$ is a path on $2n$ vertices. Hence its vertex set
can be partitioned into two independent sets of cardinality $n$. Choose such a partition
\[
V(G)\setminus\{x_i\}=A_i\sqcup B_i,
\]
and define
\[
D_i=V(G)\setminus A_i,
\qquad
E_i=V(G)\setminus B_i.
\]
Both $D_i$ and $E_i$ are minimal vertex covers of $G$, and $D_i\cap E_i=\{x_i\}.$
Let $d_i$ and $e_i$ denote the corresponding vertex-cover monomials.
Accordingly,  $\deg_{x_i}(d_ie_i)=2$ and $\deg_{x_j}(d_ie_i)=1$, where $j\neq i$.  Let
\[
c:=\prod_{x_j\in C}x_j
=x_2x_4\cdots x_{2n}x_{2n+1}.
\]
Since $C$ is a minimal vertex cover of $G$, we have $c\in G(J)$.
Fix $s\geq2$. For each $i$, define $g_i=d_ie_i c^{s-2}\in J^s.$
Let $b^{(i)}$ be the exponent vector of $g_i$. We claim that
\[
b^{(i)}_i=
\begin{cases}
2,&x_i\in U,\\[2mm]
s,&x_i\in C,
\end{cases}
\]
and, for $i\neq j$, we have
\[
b^{(i)}_j<
\begin{cases}
2,&x_j\in U,\\[2mm]
s,&x_j\in C.
\end{cases}
\]
Indeed, if $x_i\in U$, then $x_i\notin C$, and hence $b^{(i)}_i=2.$ Also, for $j\neq i$, we get
\[
b^{(i)}_j=
\begin{cases}
1,&x_j\in U,\\
s-1,&x_j\in C.
\end{cases}
\]
If $x_i\in C$, then $x_i$ occurs once in each of $d_i$ and $e_i$
and once in each copy of $c$, so $b^{(i)}_i=s.$ For $j\neq i$, we obtain
\[
b^{(i)}_j=
\begin{cases}
1,&x_j\in U,\\
s-1,&x_j\in C.
\end{cases}
\]
Consequently, if we put
\[
t_i=
\begin{cases}
2,&x_i\in U,\\[2mm]
s,&x_i\in C,
\end{cases}
\]
then $b^{(i)}_i=t_i$ and $b^{(i)}_j<t_j$ for all $i\neq j.$ This implies that $g_1,\ldots,g_{2n+1}$
form a compatible $S$-dominant witness system for $S=\{1,\ldots,2n+1\}.$

It remains to prove that the corresponding escape system has no
solution. Suppose, on the contrary, that there exists $\lambda\in\mathbb N_0^r$ with $|\lambda|=s$ such that
\[
(A\lambda)_i<t_i
\qquad
\text{for all }i=1,\ldots,2n+1.
\]
The vector $\lambda$ represents a product of $s$ minimal generators of $J$.
Every minimal vertex cover of $C_{2n+1}$ has at least $n+1$
vertices. Hence, if the selected minimal vertex covers are
$Q_1,\ldots,Q_s$, then we have
\[
\sum_{\ell=1}^s |Q_\ell|\geq s(n+1).
\]
Therefore, if  $q_i=s-(A\lambda)_i,$ then 
\begin{equation}
\sum_{i=1}^{2n+1}q_i
=
s(2n+1)-\sum_{\ell=1}^s|Q_\ell|
\leq
s(2n+1)-s(n+1)=ns.\label{inequality.1}
\end{equation} Note that $q_i$ is the number of the $s$ selected vertex covers which do not contain $x_i$. 

On the other hand, for $x_i\in U$, the inequality $(A\lambda)_i<t_i=2$ gives $(A\lambda)_i\leq1,$
and hence $q_i\geq s-1.$  Since $|U|=n$, we obtain
\begin{equation}
\sum_{x_i\in U}q_i\geq n(s-1). \label{inequality.2}
\end{equation}
For $x_i\in C$, the inequality $(A\lambda)_i<t_i=s$ gives $(A\lambda)_i\leq s-1,$ and consequently $q_i\geq1.$
Since $|C|=n+1$, we can deduce that
\begin{equation}
\sum_{x_i\in C}q_i\geq n+1. \label{inequality.3}
\end{equation}
Combining (\ref{inequality.2}) and (\ref{inequality.3}), we can conclude that
\[
\sum_{i=1}^{2n+1}q_i
\geq n(s-1)+(n+1)
=ns+1,
\]
which contradicts (\ref{inequality.1}). Thus,  the escape system $\mathcal E_{S,s}$ has no solution.
Due to  $t_i\geq1$ for each $i$, Theorem~\ref{thm:matrix_power_witness} implies that
$\langle x_1,\ldots,x_{2n+1}\rangle\in\operatorname{Ass}(R/J^s)$ for all $s\geq 2$, as claimed.
\end{exm}


We now present an example for a non-square-free monomial ideal, revisiting Example 4.7 in \cite{CNQ}.

\begin{exm} \label{Example-2} 
Assume  $I=\langle x^{2m+1}z,x^my^4,x^{m+1}y^2,y^{2m+1}z\rangle$
with $m\geq2$ is a monomial ideal in $R=\mathbb K[x,y,z]$, and let
 $\mathfrak{m}=\langle x,y,z\rangle.$ Using Theorem~\ref{thm:matrix_power_witness}, we show that
$\mathfrak{m}\in\operatorname{Ass}(R/I^s)$ for all $s\geq m$.
Let $G(I)=\{f_1,f_2,f_3,f_4\},$  where
\[
f_1=x^{2m+1}z,\qquad
f_2=x^my^4,\qquad
f_3=x^{m+1}y^2,\qquad
f_4=y^{2m+1}z.
\]
 Then  the  exponent matrix of $I$  is given by
\[
A=
\begin{pmatrix}
2m+1&m&m+1&0\\
0&4&2&2m+1\\
1&0&0&1
\end{pmatrix}.
\]

The witness vectors are chosen by first constructing a witness for $y$ and then modifying the corresponding product to obtain a witness for $x$. Indeed, since
\[
f_2=x^my^4,\qquad f_3=x^{m+1}y^2,
\]
replacing one copy of $f_2$ by one copy of $f_3$ increases the $x$-exponent by $1$ while decreasing the $y$-exponent by $2$. We therefore first consider
\[
g_y=f_2^m f_4^{s-m},
\]
which is a product of exactly $s$ generators of $I$. Replacing one copy of $f_2$ by $f_3$ then gives
\[
g_x=f_2^{m-1}f_3f_4^{s-m}.
\]
To construct a witness for $z$, we use one additional copy of $f_4$.
To keep the $x$- and $y$-exponents below the corresponding
distinguished exponents, we take $m-1$ copies of $f_3$. Thus, we set
\[
g_z=f_3^{m-1}f_4^{s-m+1}.
\]

 Fix  $s\geq m$.  All three products contain exactly $s$ factors.  Consequently, their multiplicity vectors are
\[
\lambda^{(x)}
=
\begin{pmatrix}
0\\
m-1\\
1\\
s-m
\end{pmatrix},
\qquad
\lambda^{(y)}
=
\begin{pmatrix}
0\\
m\\
0\\
s-m
\end{pmatrix},
\qquad
\lambda^{(z)}
=
\begin{pmatrix}
0\\
0\\
m-1\\
s-m+1
\end{pmatrix}.
\]
Since $s\geq m$, all three vectors belong to $\mathbb{N}_0^4$, and  it is also easy to check that
\[
\left|\lambda^{(x)}\right|
=
\left|\lambda^{(y)}\right|
=
\left|\lambda^{(z)}\right|
=s.
\]
The corresponding monomials are as follows
$$g_x=f^{\lambda^{(x)}}=f_2^{m-1}f_3f_4^{s-m}, g_y=f^{\lambda^{(y)}}=f_2^mf_4^{s-m},
 g_z=f^{\lambda^{(z)}}=f_3^{m-1}f_4^{s-m+1}.$$
In particular, it  is straightforward to verify that
\[
\begin{aligned}
g_x
&=(x^my^4)^{m-1}
  (x^{m+1}y^2)
  (y^{2m+1}z)^{s-m}\\
&=x^{m(m-1)+m+1}
  y^{4(m-1)+2+(2m+1)(s-m)}
  z^{s-m}\\
&=x^{m^2+1}
  y^{4m-2+(2m+1)(s-m)}
  z^{s-m}.
\end{aligned}
\]
Similarly, one can easily see that $g_y
=x^{m^2}
y^{4m+(2m+1)(s-m)}
z^{s-m},$ and
$$g_z
=x^{m^2-1}
y^{4m-1+(2m+1)(s-m)}
z^{s-m+1}.$$
Therefore, their exponent vectors are as follows
\[
b^{(x)}
=A\lambda^{(x)}
=
\begin{pmatrix}
m^2+1\\
4m-2+(2m+1)(s-m)\\
s-m
\end{pmatrix},
\]
\[
b^{(y)}
=A\lambda^{(y)}
=
\begin{pmatrix}
m^2\\
4m+(2m+1)(s-m)\\
s-m
\end{pmatrix},
\]
and
\[
b^{(z)}
=A\lambda^{(z)}
=
\begin{pmatrix}
m^2-1\\
4m-1+(2m+1)(s-m)\\
s-m+1
\end{pmatrix}.
\]
Set  $t_x=m^2+1,$ $t_y=4m+(2m+1)(s-m),$ and $t_z=s-m+1.$ Then
\[
b_x^{(x)}=t_x,\qquad
b_y^{(y)}=t_y,\qquad
b_z^{(z)}=t_z.
\]
Moreover, $b_y^{(x)} =4m-2+(2m+1)(s-m)<t_y$ and $b_z^{(x)}=s-m<t_z.$ Also,
 $b_x^{(y)}=m^2<m^2+1=t_x$  and  $b_z^{(y)}=s-m<s-m+1=t_z.$ Finally, one can see
 $b_x^{(z)}=m^2-1<m^2+1=t_x$ and $b_y^{(z)}=4m-1+(2m+1)(s-m)<t_y.$
  Hence,  we can deduce that
 \[
 \deg_x(g_x)>\deg_x(g_y),\deg_x(g_z),
\]
\[
\deg_y(g_y)>\deg_y(g_x),\deg_y(g_z),
\]
\[
\deg_z(g_z)>\deg_z(g_x),\deg_z(g_y).
\]
Therefore, $g_x,g_y,g_z$ form a compatible $\{1,2,3\}$-dominant witness system.

It remains to show that the corresponding escape system has no integer
solution. Suppose, on the contrary, that there exists $\lambda=(a,b,c,d)^T\in\mathbb{N}_0^4$
with $a+b+c+d=s$  such that
\[
(A\lambda)_x<t_x,\qquad
(A\lambda)_y<t_y,\qquad
(A\lambda)_z<t_z.
\]
Since all entries are integers, we obtain
\[
(2m+1)a+mb+(m+1)c\leq m^2,
\]
\[
4b+2c+(2m+1)d
\leq
4m+(2m+1)(s-m)-1, \text{ and }
\]
\[
a+d\leq s-m.
\]
The last inequality gives that $b+c=s-a-d\geq m.$ Consequently, we get
\[
\begin{aligned}
m^2
&\geq (2m+1)a+mb+(m+1)c\\
&=(2m+1)a+m(b+c)+c\\
&\geq (2m+1)a+m^2+c.
\end{aligned}
\]
It follows that  $a=c=0.$  Since $b+c\geq m$, we obtain $b\geq m.$ On the other hand, the first inequality gives
 $mb\leq m^2,$ and hence $b\leq m.$ Therefore, $b=m.$ Since $a+b+c+d=s,$ we obtain $d=s-m.$ This implies that
\[
\begin{aligned}
(A\lambda)_y
&=4b+2c+(2m+1)d\\
&=4m+(2m+1)(s-m)\\
&=t_y,
\end{aligned}
\]
which contradicts $(A\lambda)_y<t_y.$ Thus, the escape system has no solution.
Finally, we note that
\[
t_x=m^2+1\geq1,
\]
\[
t_y=4m+(2m+1)(s-m)\geq1, \text{  and  }
\]
\[
t_z=s-m+1\geq1.
\]
Hence, by Theorem~\ref{thm:matrix_power_witness}, we get $\mathfrak{m}=\langle x,y,z\rangle\in\operatorname{Ass}(R/I^s)$
for all $s\geq m$.
\end{exm}


\subsubsection{A lower bound for the omega invariant}
\label{sec_power_lower_bound}

For a fixed $n$, let $\mathcal W_n(I)$ denote the collection of subsets
$S\subseteq[k]$ for which the matrix witness criterion of
Theorem~\ref{thm:matrix_power_witness} succeeds. Thus, $S\in\mathcal W_n(I)$
if there exists a compatible $S$-dominant witness system satisfying
$t_u\geq1$ for all $u$, such that the corresponding escape system is
infeasible. Hence,  we immediately obtain the following result, generalizing Corollary \ref{cor_omega_any_monomial}.

\begin{cor} \label{cor:omega_power_lower_bound}
Let $I\subseteq R=\mathbb K[x_1,\ldots,x_k]$ be a monomial ideal. Then, for all $n\geq1$, we have
$$|\mathcal W_n(I)| \leq\omega(I^n).$$
\end{cor}

\begin{proof}

According to  Theorem~\ref{thm:matrix_power_witness}, every $S\in\mathcal W_n(I)$ determines  an associated prime
 $P_S\in\operatorname{Ass}(R/I^n).$ Note that different subsets of $[k]$ give different monomial prime ideals.
Therefore, $|\mathcal W_n(I)|
\leq
|\operatorname{Ass}(R/I^n)|
=
\omega(I^n).
$
\end{proof}

Thus, the exponent matrix of $I$ provides a computable lower bound for
the omega invariant of every power $I^n$.


In the following example we  illustrate how the escape system can show that a natural
candidate witness fails for a power.

\begin{exm} \label{Example-3} Consider the following monomial ideal
\[
I=\langle x_1^3x_2,\;x_1x_2^3,\;x_1^3x_3^2,\;
x_1^3x_3x_4,\;x_1^3x_4^2\rangle
\subseteq R=\mathbb K[x_1,x_2,x_3,x_4].
\]
Then the exponent matrix is given by
\[
A=
\begin{pmatrix}
3&1&3&3&3\\
1&3&0&0&0\\
0&0&2&1&0\\
0&0&0&1&2
\end{pmatrix}.
\]
Consider the prime $P=\langle x_1,x_2\rangle$. The first two generators have
exponent vectors $(3,1,0,0)$ and $(1,3,0,0),$ respectively. Thus, for $I$, the vectors
\[
\lambda^{(1)}
=
\begin{pmatrix}
1\\
0\\
0\\
0\\
0
\end{pmatrix}
\quad \text{ and } \quad
\lambda^{(2)}
=
\begin{pmatrix}
0\\
1\\
0\\
0\\
0
\end{pmatrix}
\]
form a compatible $\{1,2\}$-dominant witness system, with distinguished
exponents  $t_{x_1}=3$ and $t_{x_2}=3.$  We now ask whether this witness system remains effective for $I^n$.
The natural extension of this witness system to $I^n$ is given by $g_1=f_1^n$ and  $g_2=f_2^n.$
Their corresponding multiplicity vectors are
\[
\lambda^{(1)}_n=
\begin{pmatrix}
n\\
0\\
0\\
0\\
0
\end{pmatrix}
\quad \text{ and } \quad
\lambda^{(2)}_n=
\begin{pmatrix}
0\\
n\\
0\\
0\\
0
\end{pmatrix}.
\]
Thus, the distinguished exponents for this extended witness system are
 $t_{x_1}^{(n)}=3n$ and $t_{x_2}^{(n)}=3n.$ For $I^n$, a product of $n$ minimal generators has the form
 $f_1^a f_2^b f_3^c f_4^d f_5^e$ with $a+b+c+d+e=n.$ Its exponent vector is
\[
A
\begin{pmatrix}
a\\
b\\
c\\
d\\
e
\end{pmatrix}
=
\begin{pmatrix}
3a+b+3c+3d+3e\\
a+3b\\
2c+d\\
d+2e
\end{pmatrix}.
\]
Therefore, the escape system associated to this extended witness
system is
\begin{equation}
\begin{cases}
3a+b+3c+3d+3e<3n,\\
a+3b<3n,\\
a+b+c+d+e=n,\\
a,b,c,d,e\geq0.
\end{cases}
\label{61}
\end{equation}
Using $a+b+c+d+e=n$, the first inequality in \eqref{61} becomes $3n-2b<3n,$
and so $b>0$. Similarly, the second inequality becomes $3n-2a-3c-3d-3e<3n,$
and consequently $a+c+d+e>0.$ For every $n\geq2$, the following choice
\[
a=1,\qquad b=1,\qquad c=n-2,\qquad d=e=0
\]
gives an integer solution. The corresponding product is $f_1f_2f_3^{\,n-2},$ whose exponent vector is
 $(3n-2,\;4,\;2n-4,\;0).$ Both the $x_1$- and $x_2$-coordinates are strictly below the
corresponding distinguished exponents $3n$ for every $n\geq2$.
Thus, the escape system has an integer solution for every $n\geq2$.
Consequently, this particular extension of the witness system for $I$
does not remain effective for $I^n$ when $n\geq2$.

This example illustrates that a witness system for $I$ need not remain
effective for $I^n$. The matrix method detects this phenomenon without
computing the minimal generating set of $I^n$. In particular, the
existence of an escaping product does not imply that $\langle x_1,x_2\rangle\notin\operatorname{Ass}(R/I^n);$
it only shows that this particular witness system fails to certify the
associatedness of $\langle x_1,x_2\rangle$ for $I^n$.
\end{exm}



\subsubsection{The escape region}
\label{sec_escape_region}

The integer system \eqref{eq:escape_system} has a useful geometric
interpretation. For fixed $S$, $n$, and a chosen witness system, define
\[
\mathcal P_{S,n}
=
\left\{
\lambda\in\mathbb R_{\geq0}^s:
\mathbf 1^T\lambda=n,\,
(A\lambda)_{i_u}<t_u
\text{ for all }u
\right\},
\]
where 
$\mathbf{1}^T= (1,1,\ldots,1).$
We call $\mathcal P_{S,n}$ the
\emph{escape region} associated to the selected witness system. 

\begin{cor}
\label{cor:empty_escape_omega}
Let $I$ be a monomial ideal. Suppose that for distinct subsets
$S_1,\ldots,S_r\subseteq[k]$ there exist compatible
$S_i$-dominant witness systems such that $\mathcal P_{S_i,n}=\varnothing$ for $i=1,\ldots,r.$ Then
 $r\leq \omega(I^n).$
\end{cor}

\begin{proof}
For each $i$, the emptiness of $\mathcal P_{S_i,n}$ implies that $\mathcal P_{S_i,n}\cap\mathbb N_0^s=\varnothing.$
Hence, by Theorem~\ref{thm:matrix_power_witness}, we obtain $P_{S_i}\in\operatorname{Ass}(R/I^n).$
 Since the subsets $S_1,\ldots,S_r$ are distinct, the prime ideals
$P_{S_1},\ldots,P_{S_r}$ are distinct. Therefore, we get
\[
r\leq
|\operatorname{Ass}(R/I^n)|
=
\omega(I^n).
\]
\end{proof}





\subsection{A matrix extension criterion}

\label{sec:matrix_extension}

The matrix witness criterion also provides a mechanism for lifting
certified supports from one power to the next.  The basic idea is that
a witness system for $I^m$ may be extended to a witness system for
$I^{m+1}$ by adding the same generator of $I$ to each component of the
witness system. In what follows, we  formulate the corresponding matrix condition.

\begin{defn} \label{def:extension_column}
Let $S=\{i_1,\ldots,i_r\}\in\mathcal W_m(I)$ and let
 $\lambda^{(1)},\ldots,\lambda^{(r)}
\in\mathbb N_0^s$ be a compatible $S$-dominant witness system for $I^m$.  Put
\[
b^{(u)}=A\lambda^{(u)}
\qquad\text{and}\qquad
t_u=b^{(u)}_{i_u}.
\]

A column $A_{\bullet q}$ of $A$ is called an
\emph{extension column} for this witness system if the following two
conditions hold:
\begin{equation}
b^{(u)}_{i_u}+a_{i_uq}
>
b^{(u)}_{i_v}+a_{i_vq}
\qquad
\text{for all }u\neq v,
\tag{E1}
\label{eq:extension_dominance}
\end{equation}
and the integer system
\begin{equation}
\begin{cases}
\mu\in\mathbb N_0^s,\\
|\mu|=m+1,\\
(A\mu)_{i_u}<t_u+a_{i_uq},
\qquad u=1,\ldots,r
\end{cases}
\tag{$\mathcal E^{(q)}_{S,m+1}$}
\label{eq:extension_escape}
\end{equation}
has no solution.
\end{defn}

The first condition ensures that adding the same generator $f_q$ to
each component preserves the required dominance relations.  The second
condition is precisely the escape condition required for the lifted
witness system.

\begin{thm}[Matrix Extension Theorem]\label{thm:matrix_extension}
Let $I=\langle f_1,\ldots,f_s\rangle\subseteq R=\mathbb K[x_1,\ldots,x_k]$
be a monomial ideal with exponent matrix $A=(a_{ij})\in\mathbb N_0^{k\times s}.$
Fix $m\geq1$ and let $S=\{i_1,\ldots,i_r\}\subseteq[k].$ Suppose that
 $S\in\mathcal W_m(I)$ and let  $\lambda^{(1)},\ldots,\lambda^{(r)}$
be a compatible $S$-dominant witness system for $I^m$.  Put
 $b^{(u)}=A\lambda^{(u)}$ and $t_u=b^{(u)}_{i_u}.$
If $A_{\bullet q}$ is an extension column for this witness system,
then $S\in\mathcal W_{m+1}(I).$
More precisely, if $e_q$ denotes the $q$-th unit vector of
$\mathbb N_0^s$, then $\lambda^{(u)}+e_q,$ where
$u=1,\ldots,r,$ is a compatible $S$-dominant witness system for $I^{m+1}$.
\end{thm}

\begin{proof}
For each $u=1,\ldots,r$, define  $\lambda'^{(u)}:=\lambda^{(u)}+e_q.$ Since
 $|\lambda^{(u)}|=m,$ we have  $|\lambda'^{(u)}|=m+1.$ Thus,
  $\lambda'^{(u)}$ is an admissible exponent vector for a product of $m+1$ generators of $I$.
 Set $b'^{(u)}=A\lambda'^{(u)}.$  Then we obtain
\[
b'^{(u)}
=
A\lambda^{(u)}+Ae_q
=
b^{(u)}+A_{\bullet q}.
\]
Consequently,  $b'^{(u)}_{i_v}=b^{(u)}_{i_v}+a_{i_vq}$  for every $u,v$.  In particular,  we can deduce that
$b'^{(u)}_{i_u}
=
b^{(u)}_{i_u}+a_{i_uq}
=
t_u+a_{i_uq}.$  It follows from  \eqref{eq:extension_dominance} that
 $b'^{(u)}_{i_u}>b'^{(u)}_{i_v}$  for every $v\neq u.$ Hence,  the new witness system is $S$-dominant.

Next, let $u\neq v$.  Since the original witness system is compatible, this implies that
 $b^{(u)}_{i_v}<b^{(v)}_{i_v}.$  Adding the same quantity $a_{i_vq}$ to both sides gives
\[
b^{(u)}_{i_v}+a_{i_vq}
<
b^{(v)}_{i_v}+a_{i_vq}.
\]
Therefore, we get $b'^{(u)}_{i_v}<b'^{(v)}_{i_v}.$ Thus the new witness system is compatible.

The new distinguished exponents are $t'_u
=
b'^{(u)}_{i_u}
=
t_u+a_{i_uq}
\geq t_u
\geq1.
$

Finally, the escape system associated with the new witness system is
precisely
\[
\begin{cases}
\mu\in\mathbb N_0^s,\\
|\mu|=m+1,\\
(A\mu)_{i_u}<t_u+a_{i_uq},
\qquad u=1,\ldots,r.
\end{cases}
\]
By the definition of an extension column, this system has no solution.
This shows that  $\lambda'^{(1)},\ldots,\lambda'^{(r)}$ is a compatible $S$-dominant witness system for $I^{m+1}$.
By Theorem \ref{thm:matrix_power_witness}, it follows that  $P_S\in\operatorname{Ass}(R/I^{m+1}).$
Consequently,  $S\in\mathcal W_{m+1}(I).$ This finishes the proof.
\end{proof}


The preceding theorem gives a sufficient condition for lifting an
individual certified support.  Applying it separately to every
certified support at level $m$ gives the following consequence.

\begin{cor} \label{cor:omega_monotonicity}
Let $I\subseteq \mathbb K[x_1,\ldots,x_k]$ be a monomial ideal.  Suppose that for every
$S\in\mathcal W_m(I)$ there exists a compatible $S$-dominant witness system for $I^m$
admitting an extension column in the sense of Definition~\ref{def:extension_column}.  Then
 $\mathcal W_m(I)\subseteq\mathcal W_{m+1}(I),$ and consequently
 $|\mathcal W_m(I)| \leq |\mathcal W_{m+1}(I)|.$ If, in addition, the matrix witness criterion is complete for both
$I^m$ and $I^{m+1}$, then
$$\omega(I^{m+1})\geq\omega(I^m).$$
\end{cor}

\begin{proof}
Let $S\in\mathcal W_m(I).$ By hypothesis, there exists a compatible $S$-dominant witness system
admitting an extension column.  Theorem~\ref{thm:matrix_extension}
therefore gives $S\in\mathcal W_{m+1}(I).$ Hence, $\mathcal W_m(I)\subseteq\mathcal W_{m+1}(I),$
and so $|\mathcal W_m(I)|\leq|\mathcal W_{m+1}(I)|.$

Now assume that the matrix witness criterion is complete for both
$I^m$ and $I^{m+1}$.  Then  $\mathcal W_m(I)=\Sigma_m(I)$ and $\mathcal W_{m+1}(I)=\Sigma_{m+1}(I),$
 where $$\Sigma_n(I)=\{S\subseteq[k]:P_S\in\operatorname{Ass}(R/I^n)\}.$$
Thus,   $\omega(I^m)=|\Sigma_m(I)| =|\mathcal W_m(I)|$ and
 $\omega(I^{m+1})=|\Sigma_{m+1}(I)|=|\mathcal W_{m+1}(I)|.$ The inclusion
 $\mathcal W_m(I)\subseteq\mathcal W_{m+1}(I)$ therefore yields $\omega(I^m) \leq \omega(I^{m+1}).$
\end{proof}


\begin{rem}\label{rem:extension_column}
The escape condition in Definition~\ref{def:extension_column} is
essential.  In general, infeasibility of the escape system at level
$m$ does not imply infeasibility of the corresponding system at level
$m+1$ after the thresholds are increased by the entries of an
extension column.  Thus the existence of an extension column is a
genuine additional hypothesis rather than a consequence of the
existence of the original witness system.

Moreover, the compatibility inequalities require no additional
condition on the extension column: because the same generator is added
to every component of the witness system, the same quantity
$a_{i_vq}$ is added to both sides of each compatibility inequality.
\end{rem}

\subsubsection{Application to edge ideals of simple graphs} \label{sec:edge_ideal_extension}
We now specialize the matrix extension criterion to edge ideals of
finite simple graphs. Let $G=(V,E)$ be a finite simple undirected graph with
 $V=\{x_1,\ldots,x_k\},$ and let
\[
I(G)=\bigl(x_ix_j:\{x_i,x_j\}\in E\bigr)
\subseteq R=K[x_1,\ldots,x_k]
\]
be its edge ideal. Write $E=\{e_1,\ldots,e_s\}$ and $e_q=\{x_{\alpha_q},x_{\beta_q}\}.$
The exponent matrix of \(I(G)\) is $A=(a_{iq})\in\{0,1\}^{k\times s},$
where
\[
a_{iq}
=
\begin{cases}
1,&\text{if }x_i\in e_q,\\
0,&\text{otherwise}.
\end{cases}
\]
Thus each column $A_{\bullet q}=\begin{pmatrix}a_{1q} & a_{2q} & \cdots & a_{kq}\end{pmatrix}^{T}$
 is the incidence vector of an edge
of \(G\), and has exactly two entries equal to \(1\).

For $\lambda=(\lambda_1,\ldots,\lambda_s)\in\mathbb N_0^s,$  the vector
\[ A\lambda = \begin{pmatrix} a_{11} & a_{12} & \cdots & a_{1s}\\ a_{21} & a_{22} & \cdots & a_{2s}\\ \vdots & \vdots & \ddots & \vdots\\ a_{k1} & a_{k2} & \cdots & a_{ks} \end{pmatrix} \begin{pmatrix} \lambda_1\\ \lambda_2\\ \vdots\\ \lambda_s \end{pmatrix} = \begin{pmatrix} \displaystyle\sum_{q=1}^s a_{1q}\lambda_q\\ \displaystyle\sum_{q=1}^s a_{2q}\lambda_q\\ \vdots\\ \displaystyle\sum_{q=1}^s a_{kq}\lambda_q \end{pmatrix},\]
 has a natural graph-theoretic interpretation. Namely,
 $(A\lambda)_i$  is the degree of the vertex \(x_i\) in the edge multigraph obtained
from \(G\) by taking each edge \(e_q\) with multiplicity \(\lambda_q\).
Consequently, a witness system for an edge ideal may be viewed as a
collection of edge-multisets whose degree vectors satisfy the required
dominance and compatibility conditions.


\begin{prop}[Edge-extension criterion] \label{prop:edge_extension}
Let \(G\) be a finite simple graph and let \(I=I(G)\) be its edge ideal. Suppose that
$S=\{i_1,\ldots,i_r\}\in\mathcal W_m(I)$ and that $\lambda^{(1)},\ldots,\lambda^{(r)}\in\mathbb N_0^s$
is a compatible \(S\)-dominant witness system for \(I^m\). Put $b^{(u)}=A\lambda^{(u)}$ and
$t_u=b^{(u)}_{i_u}.$   Let  $e_q=\{x_p,x_\ell\}\in E(G)$ and suppose that the following dominance condition holds:
\begin{equation}
b^{(u)}_{i_u}+a_{i_uq}
>
b^{(u)}_{i_v}+a_{i_vq}
\qquad
\text{for all }u\neq v.
\tag{G1}
\label{eq:edge_extension_dominance}
\end{equation}
Assume, in addition, that the system
\begin{equation}
\begin{cases}
\mu\in\mathbb N_0^s,\\
|\mu|=m+1,\\
(A\mu)_{i_u}<t_u+a_{i_uq},
\qquad u=1,\ldots,r
\end{cases}
\tag{G2}
\label{eq:edge_extension_escape}
\end{equation}
has no solution. Then we have $S\in\mathcal W_{m+1}(I).$ More precisely, the vectors
 $\lambda'^{(u)}=\lambda^{(u)}+e_q,$ where  $u=1,\ldots,r$,  form a compatible \(S\)-dominant witness system for \(I^{m+1}\).
\end{prop}

\begin{proof}
For each \(u=1,\ldots,r\), define $\lambda'^{(u)}=\lambda^{(u)}+e_q.$ Since
 $|\lambda^{(u)}|=m,$ this implies that  $|\lambda'^{(u)}|=m+1.$ Thus, each \(\lambda'^{(u)}\) is an admissible exponent vector for a
product of \(m+1\) generators of \(I\). Moreover, we have
\[
A\lambda'^{(u)}
=
A\lambda^{(u)}+Ae_q
=
b^{(u)}+A_{\bullet q}.
\]
Consequently, $(A\lambda'^{(u)})_i=b^{(u)}_i+a_{iq}$ for every \(i\). It follows from \eqref{eq:edge_extension_dominance} that
\[
(A\lambda'^{(u)})_{i_u}
>
(A\lambda'^{(u)})_{i_v}
\qquad
\text{for every }v\neq u.
\]
Hence, the new witness system is \(S\)-dominant. Compatibility is preserved automatically. Indeed, if \(u\neq v\), then
the original witness system satisfies
\[
b^{(u)}_{i_v}<b^{(v)}_{i_v}.
\]
Adding the same quantity \(a_{i_vq}\) to both sides gives that $b^{(u)}_{i_v}+a_{i_vq}
<
b^{(v)}_{i_v}+a_{i_vq}.$ Therefore, $(A\lambda'^{(u)})_{i_v}
<
(A\lambda'^{(v)})_{i_v}.$ Thus the new witness system remains compatible. The new distinguished exponents are
 $t'_u=(A\lambda'^{(u)})_{i_u}=t_u+a_{i_uq}.$ Accordingly,  the escape system associated with the new witness system is
exactly the system in \eqref{eq:edge_extension_escape}, which is
infeasible by assumption. Therefore, $\lambda'^{(1)},\ldots,\lambda'^{(r)}$ form a compatible \(S\)-dominant witness system for \(I^{m+1}\).
This yields that  $S\in\mathcal W_{m+1}(I),$ as claimed.
\end{proof}


The preceding proposition has a direct graph-theoretic interpretation:
one adds the same edge to every member of the witness system.

\begin{rem} \label{rem:edge_extension_interpretation}
Suppose  $e_q=\{x_p,x_\ell\}.$ Then $A_{\bullet q}=e_p+e_\ell,$
where \(e_p,e_\ell\) are the corresponding standard basis vectors of
\(\mathbb N_0^k\). Thus, $b'^{(u)}=b^{(u)}+e_p+e_\ell.$
In other words, passing from \(b^{(u)}\) to \(b'^{(u)}\) increases the
degrees of precisely the two endpoints \(x_p\) and \(x_\ell\) by one.
Explicitly,
\[
b'^{(u)}_i
=
\begin{cases}
b^{(u)}_i+1,&i\in\{p,\ell\},\\
b^{(u)}_i,&i\notin\{p,\ell\}.
\end{cases}
\]
Thus the dominance condition
\eqref{eq:edge_extension_dominance} can be checked directly from the
degree vectors of the witness multigraphs.
Notice, however, that the existence of such an extension edge is not
automatic. Adding an edge may increase the degree of a competing
vertex without increasing the distinguished degree, thereby destroying
the dominance inequalities. The escape condition must also be
verified separately at level \(m+1\).
\end{rem}


\begin{thm}[Persistence criterion for certified supports of edge ideals] \label{thm:edge_ideal_persistence}
Let \(G\) be a finite simple graph and let \(I=I(G)\) be its edge ideal. Suppose that for
every $S\in\mathcal W_m(I)$ there exist a compatible \(S\)-dominant witness system
$\lambda^{(1)},\ldots,\lambda^{(r)}$ for \(I^m\) and an edge \(e_q\in E(G)\) such that
\eqref{eq:edge_extension_dominance} holds and the escape system
\eqref{eq:edge_extension_escape} is infeasible. Then  $\mathcal W_m(I) \subseteq \mathcal W_{m+1}(I).$

    If, in addition, the matrix witness criterion is complete for both
\(I^m\) and \(I^{m+1}\), then  $\omega(I^{m+1})\geq\omega(I^m).$
\end{thm}

\begin{proof}
Let $S\in\mathcal W_m(I).$ By hypothesis, \(S\) admits a compatible \(S\)-dominant witness system
and an edge \(e_q\) satisfying the hypotheses of Proposition~\ref{prop:edge_extension}. Therefore,
 $S\in\mathcal W_{m+1}(I).$ Since \(S\) was arbitrary, we get $\mathcal W_m(I) \subseteq \mathcal W_{m+1}(I).$
Now,  assume that the matrix witness criterion is complete for both
\(I^m\) and \(I^{m+1}\). Thus, we have $\mathcal W_m(I)=\Sigma_m(I)$ and
$\mathcal W_{m+1}(I)=\Sigma_{m+1}(I),$ where
\[
\Sigma_n(I)
=
\{S\subseteq[k]:P_S\in\operatorname{Ass}(R/I^n)\}.
\]
This implies that
$$\omega(I^m)=|\Sigma_m(I)|=|\mathcal W_m(I)| \; \text{ and } \;
\omega(I^{m+1})
=
|\Sigma_{m+1}(I)|
=
|\mathcal W_{m+1}(I)|.$$
The inclusion $\mathcal W_m(I)\subseteq\mathcal W_{m+1}(I)$  therefore gives that
 $\omega(I^m)\leq \omega(I^{m+1})$, as required.
\end{proof}


\begin{rem}
\label{rem:edge_ideal_caution}

The preceding theorem is a sufficient criterion for persistence of
certified supports; it does not assert that every edge ideal
automatically satisfies the extension hypothesis.
In particular, the theorem alone does not imply the unconditional
containment
\[
\operatorname{Ass}(R/I(G)^m)
\subseteq
\operatorname{Ass}(R/I(G)^{m+1}).
\]
To obtain such a conclusion from the witness criterion, one must verify
both that the criterion is complete at the relevant powers and that
every relevant support admits an extension edge satisfying the
dominance and escape conditions.
Thus the matrix extension method reduces the persistence problem to a
concrete condition on the incidence matrix of the graph.
\end{rem}

\section{Appendix}\label{sec_appendix}

In this appendix we present the Macaulay2 code for determining witnesses via Theorem \ref{thm_witness_any_monomial}.

\begin{tiny}
\begin{verbatim}

nonemptySubsets = n -> drop(subsets toList(0..n-1), 1);

exponentVectors = I -> flatten apply(flatten entries mingens I, exponents);

listIdeal = (e,R) -> monomialIdeal apply(e, i -> (gens R)#i);

condition1 = (G, p, g) -> all(0..#p-1, i -> (
        ei = (G#(g#i))#(p#i);
        ej = apply(select(0..#p-1, j -> j =!= i), j -> (G#(g#j)#(p#i)));
        ei > max ej
    )
);

condition2 = (G, p, g) -> (
    c = toList(set toList(0..#G-1) - set g);
    all(c, j -> any(toList(0..#p-1), i -> (G#j)#(p#i) >= (G#(g#i))#(p#i)))
);

constructWitness = (R, G, g, p) -> (
    n := numgens R;
    E := apply(toList(0..n-1), j -> (
        if member(j, p) then ( -- if that variable was chosen for the prime p...
            pos := position(p, k -> k == j); -- find the corresponding generator...
            (G#(g#pos))#j - 1  -- set the exponent of that variable
        ) else (
            max apply(G, f -> f#j)
        )
    ));
    product apply(gens R, E, (var, e) -> var^e)
);

witnesses = (I, R) -> (
    n = numgens R;
    P := nonemptySubsets n; -- generate list of possible primes
    G := exponentVectors I; -- generate exponent vectors of generators of I
    results := {}; -- initialize storage of results

    apply(P, p -> (
        if (#p==1) then (
            if (min apply(G, f -> f#(p#0))>=1) then (
                E = apply(toList(0..n-1), i -> (
                    ei := apply(G, f -> f#i);
                    if (i==p#0) then (
                        (min ei) - 1
                    ) else (
                        max ei
                    )
                ));
                w = product apply(gens R, E, (var, e) -> var^e);
                results = append(results, {w, listIdeal(p, R)});
            )
        ) else (
        S := subsets(toList(0..#G-1), #p); -- compute all subsets of size #p to index G
        scan(S, s -> ( -- scan through subsets of size #p to find generators
            L := permutations s; -- find all permutations of a chosen list
            v := select(1, L, g -> condition1(G, p, g) and condition2(G, p, g)); 
            -- check if any permutation passes conditions
            if #v > 0 then (
                results = append(results, {constructWitness(R, G, v#0, p), listIdeal(p, R)});
                break;
            );
        )));
    ));
    results -- return results
);

\end{verbatim}
\end{tiny}

The Macaulay2 output for Example \ref{exm_assoc} is the following; we start at line i8 since in the previous lines are typed the functions defined above.\\

i8 : R = QQ[x\_0,x\_1,x\_2]

\medskip

o8 = $R$

\medskip

o8 : PolynomialRing

\medskip

i9 : I = ideal(x\_0\textasciicircum3*x\_1\textasciicircum2*x\_2\textasciicircum3,x\_0\textasciicircum2*x\_1\textasciicircum3*x\_2,x\_0*x\_1\textasciicircum4,x\_2\textasciicircum4)

\medskip

o9 = ideal$(x_0^3x_1^2x_2^3,x_0^2x_1^3x_2,x_0x_1^4,x_2^4)$

\medskip

o9 : Ideal of $R$

\medskip

i10 : witnesses(I,R)

\medskip

o10 = \{\{$x_1^4x_2^3$, monomialIdeal$(x_0,x_2)$\}, \{$x_0^3x_1^3$, monomialIdeal$(x_1,x_2)$\}, \{$x_0x_1^3x_2^3$, monomialIdeal$(x_0,x_1,x_2)$\}\}

\medskip

o10 : List

Note that the first witness, $x_1^4x_2^3$, corresponding to the associated prime $\langle x_0,x_2\rangle$ is the same as the one we obtained in Example \ref{exm_assoc}, but the other witnesses are different.


\vspace{1.2cm}

\noindent\textbf{Acknowledgements}\\

All computations were performed using Macaulay2, a (free) software system for research in algebraic geometry (\cite{GrSt}).

\bigskip
\noindent\textbf{Funding Declaration}\\

The authors received no financial support for the research, authorship, or publication of this article.

\bigskip

\noindent\textbf{Data availability statement}\\

This manuscript has no associated data.

\bigskip

\noindent\textbf{Disclosure statement}\\

The authors declare no financial or non-financial competing interests.

\bigskip

\noindent\textbf{AI Declaration}\\

Except for the correction and optimization of the lines of the Macaulay2 code presented in the Appendix, when Google Gemini was used, the authors declare that no generative artificial intelligence (GAI) tools were used anywhere else in the preparation, writing, analysis, or publication of this manuscript.

\renewcommand{\baselinestretch}{1.0}
\small\normalsize 
\bigskip
\bibliographystyle{amsalpha}

\end{document}